%% file: GGR-part1-arXiv-25.04.06.tex
\documentclass[11pt]{amsart}
\usepackage{amssymb,latexsym,amsmath,tikz-cd,tikz}
\usepackage{enumitem}
\usepackage[mathscr]{eucal}
\usepackage{bbm,xcolor,comment}
\usepackage{mathdots}
\usepackage{chngcntr}

\numberwithin{equation}{section}
\numberwithin{figure}{subsection}

\input{macros.tex}
\input{thms.tex}

\theoremstyle{plain}
\newtheorem*{thm-intro}{Theorem${}^\dagger$}

\def\fibre{A}
\def\monzero{\Gamma_\infty}

\def\olB{{\overline{B}}}

\def\cFe{\cF_\mathrm{e}}
\def\sfh{\mathsf{h}}
\def\ellt{\frac{\log t}{2\pi\bi}}
\def\Le{\Lambda_\mathrm{e}}
\def\olm{\overline{\mu}}
\def\oln{\overline{\nu}}
\def\olP{{\overline{\wp}}}
\def\tPhi{\widetilde\Phi}
\def\PhiSBB{\Phi_\mathrm{SBB}} 
\def\sPhiSBB{\Phi_\mathrm{SBB}'}
\def\Phipt{\Phi_\fibre} 
\def\tsum{{\textstyle{\sum}}}
\def\olU{\overline{\sU}}
	
\def\CKS{MR840721}
\def\CMSP{MR3727160}

\begin{document}
%----------------------------------------------------------

\title[Analog of SBB]{Analog of Satake-Baily-Borel for period maps}

\author[Green]{Mark Green}
\email{mlg@math.ucla.edu}
\address{UCLA Mathematics Department, Box 951555, Los Angeles, CA 90095-1555}

\author[Griffiths]{Phillip Griffiths}
\email{pg@math.ias.edu}
\address{Institute for Advanced Study, 1 Einstein Drive, Princeton, NJ 08540}
\address{University of Miami, Department of Mathematics, 1365 Memorial Drive, Ungar 515, Coral Gables, FL  33146}

\author[Robles]{Colleen Robles}
\email{robles@math.duke.edu}
\address{Mathematics Department, Duke University, Box 90320, Durham, NC  27708-0320} 
\thanks{Robles is partially supported by NSF DMS 1611939, 1906352.}

\date{\today}

\begin{abstract}
We propose an analog of the Satake--Baily--Borel compactification and Borel's extension theorem for arbitrary period maps.  The proposed analog is constructed as a proper topological completion of the period map.  It is conjectured that the construction is projective algebraic, and the conjecture is reduced to a certain extension problem.
\end{abstract}
\maketitle
%----------------------------------------------------------
\setcounter{tocdepth}{1}
%----------------------------------------------------------

%----------------------------------------------------------
\section{Introduction} 
%----------------------------------------------------------

%----------------------------------------------------------
\subsection{The setup} \label{S:setup}
%----------------------------------------------------------

Fix a smooth projective variety $\olB$ with simple normal crossing divisor $Z$.  Suppose the complement $B = \olB \bs Z$ admits a variation of (pure) polarized Hodge structure with local system
\[
\begin{tikzcd}[row sep=small,column sep=tiny]
  \mathbb{V} \arrow[r,equal] \arrow[d] &
  \tilde B \times_{\pi_1(B)} V_\bZ 
  \\ B &
\end{tikzcd}\]
and Hodge bundles
\begin{equation}\label{iE:vhs}
  \cF^p \ \subset \ 
  \cV \ = \ \tilde B \times_{\pi_1(B)} V_\bC \,.
\end{equation}
Here $V_\bZ$ is a lattice, $\tilde B \to B$ is the universal cover,  $\rho : \pi_1(B) \to \tGL(V_\bZ)$ is the monodromy representation, and $V_\bC = V_\bZ \ot_\bZ \bC$.  Let
\begin{equation} \label{iE:Phi}
  \Phi : B \ \to \ \Gamma \bs D 
\end{equation}
be the induced period map.  Here $D$ is a period domain parameterizing pure, weight $\sfn$, $Q$--polarized integral Hodge structures on $V_\bZ$; and $\Gamma = \rho(\pi_1(B))$ is the image of the monodromy representation.  Applying a Tate-twist if necessary, we may assume that the Hodge structures parameterized by $D$ are effective.  Without loss of generality the period map \eqref{iE:Phi} is proper \cite{MR0259958}.  The image
\[
  \wp \ = \ \Phi(B)
\]
is a quasi-projective variety \cite{MR4557401}.

In the special case that $D$ is hermitian symmetric, the period space $\Gamma \bs D$ is also quasi-projective, with projective completion given by the Satake--Baily--Borel (SBB) compactification $(\Gamma \bs D)^\ast$ \cite{MR0216035}.  And Borel's extension theorem \cite{MR0338456} yields an extension $\Phi^* : \olB \to (\Gamma\bs D)^\ast$ of the period map \eqref{iE:Phi}.  While the SBB compactification is defined without reference to Hodge theory, a posteriori one sees that the boundary points of $(\Gamma\bs D)^*$ encode Hodge theoretic data: the period map is asymptotically approximated by a nilpotent orbit in a neighborhood of $b \in Z$, and $\Phi^*(b)$ parameterizes certain quotient Hodge structures obtained from the nilpotent orbit.  

Our goal is to define an analog of $\Phi^*$ for arbitrary period maps.  

%----------------------------------------------------------
\subsection{Topological completion}
%----------------------------------------------------------

The desired completion should be obtained from Hodge theoretic data at infinity; that is, from the limiting mixed Hodge structures of the period map.  Write 
\[
  Z \ = \ Z_1 \,\cup\, Z_2 \,\cup\cdots\cup\, Z_\nu \,,
\]
with smooth irreducible components $Z_i$.  We denote by 
\[
  Z_I \ = \ \bigcap_{i\in I} Z_i
\]
the closed strata, and
\[
  Z_I^* \ = \ Z_I \,-\, \cup_{j\not\in I} Z_j 
\] 
the smooth strata.  As we approach a point $b \in Z_I^*$ the period map $\Phi$ degenerates to a limiting mixed Hodge structure $(W,F)$ that is polarized by a cone $\s_I$ of nilpotent operators (arising from logarithms of the local monodromy around $b$).  Both the weight filtration $W$ and cone $\s_I$ are independent of our choice of $b \in Z_I^*$.  The Hodge filtration $F \in \check D$ will vary along $Z_I^*$, and is well-defined only up to the action of $\exp(\bC\s_I)$ on the compact dual $\check D$.  Nonetheless, the induced Hodge filtration $F^p(\tGr^W_a)$ on the graded quotient $\tGr^W_a = W_a/W_{a-1}$ is well-defined.  In this way we obtain a period map 
\begin{equation}\label{E:P0I}
  \Phi_I : Z_I^* \ \to \ \Gamma_I\backslash D_I \,, 
\end{equation}
cf.~\S\ref{S:PhiIgr}.

\begin{thm-intro}[Theorem \ref{T:top}]
The maps \eqref{E:P0I} can be patched together to define an extension 
\begin{equation} \label{iE:Phi0}
  \PhiSBB : \olB \ \to \ \olP
\end{equation}
of $\Phi: B \to \wp$.  The image $\olP$ is a Hausdorff topological space compactifying $\Phi(B)=\wp$, and is a finite union (not necessarily disjoint) of quasi-projective varieties.  The map $\PhiSBB$ is continuous and proper, and the fibres are projective algebraic varieties.
\end{thm-intro}

\noindent (In the interest of conciseness, many the results discussed in this introduction are stated imprecisely and/or incompletely; this is indicated by the superscript ${}^\dagger$.  The reader will find the formal statements, with all necessary definitions, in the body of the paper.)

\begin{remark} \label{R:herm}
The completion \eqref{iE:Phi0} does generalize SBB and Borel's extension theorem: when $D$ is hermitian, $\olP$ is the closure of $\wp$ in $(\Gamma \bs D)^*$, and $\PhiSBB = \Phi^*$.
\end{remark}

%----------------------------------------------------------
\subsection{Conjectural algebraic structure}
%----------------------------------------------------------

One would like to know that $\olP$ is algebraic and $\PhiSBB$ is a morphism.  (And this is the case when $D$ is hermitian, by Remark \ref{R:herm}.)  At this time, we make the slightly weaker Conjecture \ref{conj:GGLR} that a finite cover of $\PhiSBB$ is algebraic.

\begin{thm-intro}[Theorem \ref{T:hattop}]
The Stein factorization 
\[
  \begin{tikzcd}
  B \arrow[r,"\Phi'"'] \arrow[rr,bend left,"\Phi"] 
  & \wp' \arrow[r] & \wp
  \end{tikzcd}
\]
of the period map extends to a ``Stein factorization'' 
\[
  \begin{tikzcd}
  \olB \arrow[r,"\PhiSBB'"'] \arrow[rr,bend left,"\PhiSBB"] 
  & \olP' \arrow[r] & \olP
  \end{tikzcd}
\]
of \eqref{iE:Phi0}. \emph{The fibres of $\PhiSBB'$ are connected, and the fibres of $\olP' \to \olP$ are finite.  And, as in Theorem \ref{T:top}, the image $\olP'$ is a Hausdorff topological space compactifying $\wp'$, and a finite union of normal projective varieties.}
\end{thm-intro}

\begin{conjecture}[{\cite{GGLR}}] \label{conj:GGLR}
The image $\sPhiSBB (\olB) = \olP'$ is projective algebraic, and the extension $\sPhiSBB:\olB \to \olP'$ is a morphism.
\end{conjecture}

\begin{theorem}[{\cite{GGLR}}]
If $\tdim\,B=2$,
%% and the period map \eqref{iE:Phi} satisfies local Torelli, 
then Conjecture \ref{conj:GGLR} holds. 
\end{theorem}

\noindent For more on the $\tdim\,B=2$ case, see \cite{GGdim2}.

%----------------------------------------------------------
\subsection{A semi-ampleness conjecture}
%----------------------------------------------------------

Bakker--Brunebarbe--Tsimerman showed that the line bundle
\[
   \Xi  \ = \ 
   \tdet(\cF^\sfn) \ot \tdet(\cF^{\sfn-1}) \ot \cdots \ot 
  \tdet(\cF^\sfk)
  \,,\quad \sfk \,=\, {\lceil (\sfn+1)/2 \rceil} 
\]
is semi-ample over $B$, and that $\wp = \mathrm{Proj}\left( \op_d\,H^0(B,\Xi^{\ot d}) \right)$ \cite{MR4557401}.  Conjecture \ref{conj:GGLR} is closely related to 

\begin{conjecture}[{\cite{GGLR}}] \label{conj:semiample}
Deligne's extension $\Xi_\mathrm{e}$ is semi-ample over $\olB$.
\end{conjecture}

\begin{remark}
Conjecture \ref{conj:semiample} holds when $\tdim\,\wp = 1$; and when $\tdim\,B=2$ \cite{GGLR}.
\end{remark}

\begin{remark}
If Conjecture \ref{conj:semiample} holds, then the extension $\PhiSBB: \olB \to \olP$ factors as 
\[ \begin{tikzcd}
  \olB \arrow[r,"\PhiSBB'"] & 
  \olP' \arrow[r] & 
  \mathrm{Proj}\left( \op_d\,H^0(\olB,\Xi^{\ot d}) \right) \arrow[r]
  & \olP \,.
\end{tikzcd} \]
In particular, if the conjecture holds, then it would be natural to redefine $\PhiSBB$ as the map $\olB \to \mathrm{Proj}\left( \op_d\,H^0(\olB,\Xi^{\ot d}) \right)$.
\end{remark}

Relative to Conjecture \ref{conj:semiample}, we are able to prove the following.  Set $\sfk = \lceil (\sfn+1)/2 \rceil$.  Let $\cFe^p \to \olB$ denote Deligne's extension \cite{MR1416353} of the Hodge vector bundles $\cF^p \to B$.    Define
\[
  \Le \ = \ \left\{
    \begin{array}{ll}
    \tdet(\cFe^\sfn) \,=\, \tdet(\cFe^1) \,,\quad & \sfn=1 \,,\\
    \tdet(\cFe^\sfn) \ot \tdet(\cFe^{\sfn-1}) \ot \cdots \ot 
    	\tdet(\cFe^\sfk) \,,\quad & 
    \sfn  \hbox{ even,} \\
    \left[\tdet(\cFe^\sfn) \ot \tdet(\cFe^{\sfn-1}) 
    \ot \cdots \ot \tdet(\cFe^{\sfk+1}) \right]^{\ot2} \, 
    	\ot \tdet(\cFe^\sfk)\, \,, \  &
    \sfn \ge 3 \hbox{ odd.}
    \end{array}
  \right.
\]
(Both the definition of $\Le$, and that of $\Xi_\mathrm{e}$, assume that the period domain $D$ parameterizes \emph{effective} Hodge structures.)

\begin{thm-intro}[Theorem \ref{C:descends}]
A power of $\Le \to \olB$ descends to $\olP'$.
\end{thm-intro}

\noindent Corollary \ref{C:descends} is a consequence of: (i) the existence of an open cover $\{X\}$ of $\olB$ with the property that $\PhiSBB|_X$ is proper (Corollary \ref{C:top}); (ii) strong constraints on the monodromy $\monzero$ of the variation of Hodge structure over $B \cap X$ (\S\ref{S:monpt}); and (iii) a subtle arithmetic property (Theorem \ref{T:chi-infty}) asserting that the values of a certain character are roots of unity.

\begin{remark} \label{R:2q}
The power ${}^{\ot 2}$ is necessary in the definition of $\Le$ when $\sfn \ge 3$ is odd; the line bundle $\Xi_\mathrm{e}$ need not descend.  Cf.~Remark \ref{R:nopower}.
\end{remark}

\begin{remark}
It is an open question whether or not $\Le$ descends to $\olP$.
\end{remark}

Finally, we are able to reduce Conjecture \ref{conj:GGLR} to an extension problem:

\begin{thm-intro}[Theorem \ref{T:conj}]
Conjecture \ref{conj:GGLR} holds if and only if the holomorphic functions on $Z_I \cap X$ extend to holomorphic functions on $X$, for all $X \in \{X\}$.
\end{thm-intro}

%----------------------------------------------------------
\subsection*{Acknowledgements}
%----------------------------------------------------------

We have benefited from illuminating conversations and correspondence with several colleagues.  We thank Radu Laza for many conversations, and the closely related collaboration \cite{GGLR}.  With respect to the discussion of the conjecture in \S\ref{S:dconj}, thanks are due to Gregory Pearlstein for pointing out the role of equivalence relations in such questions, and to Daniel Greb for pointing out the  Moishezon structure and the related \cite[Corollary 1.3]{Greb2020}.  We are indebted to Haohua Deng for pointing out Lemmas \ref{L:nbd-Z_I} and \ref{L:top}.  Special thanks are due to Ben Bakker who provided an illuminating counter-example (\S\ref{S:ben}) to a claim made in an earlier version of the paper (Remark \ref{R:nopower}).  The third author also thanks Farid Hosseinijafari for discussions related to the example.

%----------------------------------------------------------
\tableofcontents %\listoftables
%----------------------------------------------------------

%----------------------------------------------------------
\section{Review of local behavior at infinity} \label{S:background}
%----------------------------------------------------------

Here we set notation and review well-known properties of period maps and their local behavior at infinity.  Good references for this material include \cite{\CMSP, \CKS, MR2918237, MR0259958, MR0382272}.

%----------------------------------------------------------
\subsection{Group notation}\label{S:not-gp}
%----------------------------------------------------------

Given a ring $\bZ \subset R \subset \bC$, define $V_R = V_\bZ \ot_\bZ R$.  The polarization is a nondegenerate bilinear form $Q : V_\bQ \times V_\bQ \to \bQ$ satisfying
\[
  Q(u,v) \ = \ (-1)^\sfn Q(v,u) \,,
  \quad\hbox{for all}\quad u,v \in V_\bQ \,.
\]
Let $\tGL(V_R) \simeq \tGL_r(R)$ be the group of invertible $R$--linear maps $V_R \to V_R$.  Define
\[
  G_R \ = \ \tGL(V_R,Q)
  \ = \ 
  \{ g \in \tGL(V_R) \ | \ Q(gu,gv) = Q(u,v) \,,\ 
  \forall \ u,v \in V_R\}
\]
We have
\[
  \Gamma \ \subset \ G_\bZ \,.
\]
Let $\fgl(V_R) \simeq \fgl_r(R)$ be the Lie algebra of $R$--linear maps $V_R \to V_R$.  Set
\[
  \fg_R \ = \ \fgl(V_R,Q) \ = \ 
  \{ \xi \in \fgl(V_R) \ | \ 
  0 = Q(\xi u, v) +Q(u,\xi v) \,,\ \forall \ u,v \in V_R \} \,.
\]
When $R = \bR,\bC$, $G_R$ is a Lie group with Lie algebra $\fg_R$.

%----------------------------------------------------------
\subsection{Period maps at infinity} \label{S:vhs}
%----------------------------------------------------------

%----------------------------------------------------------
\subsubsection{}
%----------------------------------------------------------

Let 
\[
  \Delta \ = \ \{ t \in \bC \ | \ |t|<1\}
\]
denote the unit disc, and 
\[
  \Delta^* \ = \ \{ t \in \bC \ | \ 0<|t|<1\} 
\]
the punctured unit disc.  The upper half plane
\[
  \sH \ = \ \{ z \in \bC \ | \ \tIm\,z>0 \}
\]
is the universal cover of $\Delta^*$, with covering map
\[
  \sH \ \to \ \Delta^*
  \quad\hbox{sending}\quad
  z \ \mapsto \ t = e^{2\pi\bi z} \,.
\]
Let 
\[
  \ell(t) \ = \ \ellt 
\]
denote the multivalued inverse.

%----------------------------------------------------------
\subsubsection{} \label{S:not-U}
%----------------------------------------------------------

Fix a point $b \in Z_I^* \subset \olB$.  We may choose a local coordinate chart 
\[
  (t,w) : \olU \,\subset\,\olB \ \stackrel{\simeq}{\longrightarrow} \ 
  \Delta^{k+r}
\]
centered at a point $b$ with 
\[
  (t,w) : \sU \,=\,B \,\cap\,\olU \ \stackrel{\simeq}{\longrightarrow} \ 
  (\Delta^*)^k \times \Delta^r \,.
\]
Without loss of generality $\olU \cap Z_I = \olU \cap Z_I^*$.  After reindexing the $Z_i$ if necessary, we have
\[
  \olU \,\cap\, Z_i \ = \ \{ t_i=0\} \,,
  \quad \hbox{for all}\quad 1 \le i \le k \,,
\]
and $\olU \,\cap\, Z_\mu = \emptyset$ for all $k+1 \le \mu \le \nu$.

%----------------------------------------------------------
\subsubsection{} \label{S:not-cone}
%----------------------------------------------------------
The counter-clockwise generator $\a_i \in \pi_1(\Delta^*) \inj \pi_1((\Delta^*)^k) = \pi_1(\sU)$ induces a quasi-unipotent monodromy operator $\gamma_i \in \tGL(V,Q)$, $1 \le i \le k$ \cite{MR0382272}.  Let $\gamma_i = \gamma_{i,s}\,\gamma_{i,u}$ be the Jordan decompostion, and  
\[
  N_i \ = \ \log\gamma_{u,i} \ \in \ \fg
\]
the nilpotent logarithm the unipotent factor.  Set
\[
  \s_I \ = \ \tspan_{\bQ_{>0}}\{ N_1 , \ldots , N_k\} 
  \ \subset \ \fg  \,,
\]
the \emph{monodromy cone (at $b$)}.

%----------------------------------------------------------
\subsubsection{} \label{S:not-loclift}
%----------------------------------------------------------

The universal cover of $\sU$ is $\widetilde \sU = \sH^k \times \Delta^r$.  The local lift
\[
  \tPhi : \widetilde\sU \ \to \ D
\]
of $\left.\Phi\right|_{\sU}$ is of the form
\begin{equation}\label{E:tPhi}
  \tPhi(z,w) \ = \ \exp( \tsum \ell(t_i) N_i ) 
  \,\tilde g(t,w) \cdot F \,.
\end{equation}
Here, $F$ is an element of the compact dual $\check D \supset D$ (which is the flag variety parameterizing the filtrations $F^p(V_\bC)$ that satisfy the first Hodge--Riemann bilinear relation 
\begin{equation}\label{E:QF}
  Q(F^p,F^q) \ = \ 0 \,,\quad \forall \ p+q > \sfn
\end{equation}
but not necessarily the second); the group $G_\bC$ acts transitively on $\check D$, and 
\begin{equation}\label{E:tilde-g}
  \tilde g : \olU \ \to \ G_\bC
\end{equation}
is a holomorphic map; and we abuse notation by conflating the multi-valued $\ell(t_i)$ with the coordinates $z_i$ on $\sH^k$.

%----------------------------------------------------------
\subsection{Limiting mixed Hodge structures} \label{S:LMHS}
%----------------------------------------------------------

%----------------------------------------------------------
\subsubsection{} 
%----------------------------------------------------------

A nilpotent operator $N\in\s_I$ determines a rational, increasing filtration $W_0 \subset W_1 \subset \cdots \subset W_{2\sfn} = V_\bQ$.  This is the unique filtration satisfying the conditions \eqref{SE:W(N)}: first, 
\begin{subequations}\label{SE:W(N)}
\begin{equation} \label{E:W(N)}
  N(W_\ell) \ \subset \ W_{\ell-2} \,.
\end{equation}
If we set $\tGr^W_\ell(V) = W_\ell/W_{\ell-1}$, then \eqref{E:W(N)} implies that there is a well-defined map $N: \tGr^W_\ell \to \tGr^W_{\ell-2}$.  The map 
\begin{equation}
  N^k : \tGr^W_{\sfn+k} \ \to \ \tGr^W_{\sfn-k}
  \quad\hbox{is an isomorphism for all } k \ge 0 \,.
\end{equation}
\end{subequations}

%----------------------------------------------------------
\subsubsection{} \label{S:not-LMHS}
%----------------------------------------------------------

Any two $N , N' \in \s_I$ determine the same filtration $W$, and we call $W = W(\s_I)$ the \emph{weight filtration} of the monodromy cone.  This weight filtration is $Q$--isotropic
\begin{subequations}\label{SE:QW}
\begin{equation}\label{E:QW}
  Q(W_k,W_\ell) \ = \ 0 \,,\quad \forall \ k+\ell < 2 \sfn \,.
\end{equation}
In particular, the polarization identifies
\begin{equation}\label{E:QW*}
  \tGr^W_\ell(V)^\vee \simeq \tGr^W_{2\sfn-\ell}(V) \,.
\end{equation}
\end{subequations}
If $F(w) = \tilde g(0,w)\cdot F$, then $(W,F(w))$, is a mixed Hodge structure (MHS) polarized by the local monodromy cone $\s_I$.  This means that $F(w)$ induces a weight $\ell$ Hodge structure $\tGr^W_\ell$; for each $N \in \s_I$, the kernel
\[
  P_{\sfn+k}(N) \ = \ 
  \tker\{ N^{k+1} : \tGr^W_{\sfn+k} \to \tGr^W_{\sfn-k-2}\}
\]  
is a Hodge substructure; and the weight $\sfn+k$ Hodge structure on $P_{\sfn+k}(N)$ is polarized by $Q(\cdot , N^k \cdot )$.
The triple $(W,F(w),\s_I)$ is a \emph{limiting} (or \emph{polarized}) \emph{mixed Hodge structure} (LMHS); and we say $\s_I$ \emph{polarizes} $(W,F(w))$.  

%----------------------------------------------------------
\subsubsection{}  \label{S:MI}
%----------------------------------------------------------

Define
\[
  M_I \ = \ \{ F \in \check D \ | \ (W,F) 
  \hbox{ is a MHS polarized by } \s_I \} \,.
\]
The stabilizer $P_W \subset G$ of $W$ is a parabolic subgroup of $G$.  The unipotent radical $P_W^{-1} \subset P_W$ is the subgroup acting trivially on $\tGr^W_\ell$, for all $\ell$.  Since $W = W(\s_I)$, the centralizer
\[
  C_I \ = \ \{ g \in G \ | \ \tAd_g(N) = N \,,\ \forall \ N \in \s_I \}
\] 
of the cone $\s_I$ is a subgroup of $P_W$.  Set $C_I^{-1} = C_I \,\cap\, P_W^{-1}$.  The group
\[
  G_I \ = \ (C_{I,\bR}/C_{I,\bR}^{-1}) \ltimes C_{I,\bC}^{-1}
\]
acts transitively $M_I$ \cite{MR3474815}.  

%----------------------------------------------------------
\subsubsection{} \label{S:PhiI}
%----------------------------------------------------------

The map 
\begin{equation}\label{E:FI}
  F_I \,:\, Z_I^* \cap \olU \ \to \ M_I \,,\quad
  w \mapsto F_I(w) = \tilde g(0,w) \cdot F
\end{equation}
defines a variation of limiting mixed Hodge structure $(W,F_I(w),\s_I)$ over $Z_I^* \cap \olU$.  The map \eqref{E:FI} is not well-defined; it depends on our choice of coordinates $(t,w)$.  What is well-defined is the composition 
\[
\begin{tikzcd}
  Z_I^* \cap \olU 
  \arrow[r,"F_I"] & 
  M_I \arrow[r,->>,"\nu_I"] &
  \exp(\bC\s_I)\bs M_I \,.
\end{tikzcd}
\]
(It is the nilpotent orbit that is well-defined.)  This yields a map 
\begin{equation} \label{E:PsiI}
  \Psi_I : Z_I^* \ \to \ (\exp(\bC\s_I) \Gamma_I) \bs M_I\,.
\end{equation}
Here $\Gamma_I \subset C_I \cap \Gamma$ is the monodromy of the variation of limiting mixed Hodge structure along $Z_I^*$.

%----------------------------------------------------------
\subsubsection{} \label{S:PhiIgr}
%----------------------------------------------------------

The quotient space 
\begin{equation}\label{E:DI}
  D_I \ = \ C_{I,\bC}^{-1} \bs M_I
\end{equation}
is a product (over $0 \le k \le \sfn$) of period domains parameterizing weight $\sfn+k$ Hodge structures on $P_{\sfn+k}(N)$ that are polarized by $Q(\cdot , N^k \cdot )$.

Let 
\[
  P_W^{-2} \ = \ \{ g \in P_W \ | \ g 
  \hbox{ acts trivially on } W_\ell/W_{\ell-2}\,,\ 
  \forall \ \ell \} \ \subset \ P_W^{-1} \,,
\]
and set 
\[
  C_I^{-2} \ = \ C_I \,\cap\, P_W^{-2} \,.
\]
The condition \eqref{E:W(N)} implies that $\exp(\bC\s_I) \subset C_{I,\bC}^{-2}$.  So the quotient map $M_I \sur D_I$ descends to $\exp(\bC_I)\bs M_I \sur D_I$, and we obtain a commuting diagram 
\[
  \begin{tikzcd}
    Z_I^* \arrow[r,"\Psi_I"]  \arrow[rd,"\Phi_I"']
    & (\exp(\bC\s_I) \Gamma_I)\bs M_I \arrow[d,->>]\\
    & \Gamma_I \bs D_I
  \end{tikzcd}
\]
that defines the map $\Phi_I$ of \eqref{E:P0I}.

%----------------------------------------------------------
\subsection{Deligne splittings} \label{S:not-ds}
%----------------------------------------------------------

%----------------------------------------------------------
\subsubsection{} \label{S:ds1}
%----------------------------------------------------------

Given a mixed Hodge structure $(W,F)$ on $V$, we have a Deligne splitting
\[
  V_\bC \ = \ \op\, V^{p,q}_{W,F} 
\]
satisfying
\begin{equation} \label{E:ds-WF}
  W_\ell \ = \ \bigoplus_{p+q\le \ell} V^{p,q}_{W,F} \tand
  F^k \ = \ \bigoplus_{p\ge k} V^{p,q}_{W,F} \,,
\end{equation}
and
\begin{equation}\label{E:ds-conj}
  \overline{V^{p,q}_{W,F}} \ \equiv \ V^{q,p}_{W,F}
  \quad \hbox{mod} \quad \bigoplus_{r<p,q<s} V^{r,s}_{W,F} \,.
\end{equation}
The mixed Hodge structure is \emph{$\bR$--split} if equality holds in \eqref{E:ds-conj}.  

If $(W,F,N)$ is a limiting mixed Hodge structure, then 
\begin{equation}\label{E:NVpq}
  N(V^{p,q}_{W,F}) \ \subset \ V^{p-1,q-1}_{W,F} \,,
\end{equation}
and the map
\[
  N^k : V^{p,q}_{W,F} \ \to \ 
  V^{p-k,q-k}_{W,F}
\]
is an isomorphism, for all $p+q=n+k$.  The decomposition is $Q$-orthogonal in the sense that
\begin{equation}\label{E:ds-Q}
  Q(V^{p,q}_{W,F} \,,\, V^{r,s}_{W,F}) \ = \ 0 \,,
  \quad\hbox{for all}\quad (p+r,q+s) \not= (\sfn,\sfn) \,.
\end{equation}

%----------------------------------------------------------
\subsubsection{}
%----------------------------------------------------------

A limiting mixed Hodge structure $(W,F,N)$ on $V$ induces one on the Lie algebra $\fg$
\begin{eqnarray*}
  F^p(\fg) & = & \{ \x \in \fg_\bC \ | \ \xi(F^k) \subset F^{k+p} 
  \,,\ \forall \ k \} \\
  W_\ell(\fg) & = & \{ \xi \in \fg \ | \ 
  \xi(W_k) \subset W_{k+\ell} \,,\ \forall \ k \} \,.
\end{eqnarray*}
The Deligne splitting
\begin{subequations}\label{SE:gpq}
\begin{equation}
  \fg_\bC \ = \ \op\,\fg^{p,q}_{W,F} \,,
\end{equation}
is given by 
\begin{equation} \label{E:gpq-dfn}
  \fg^{p,q}_{W,F} \ = \ 
  \{ \x \in \fg_\bC \ | \ \x(V^{r,s}_{W,F}) \subset V^{p+r,q+s}_{W,F}\,, 
  \ \forall \ r,s \} \,,
\end{equation}
%satisfies
%\begin{equation}\label{E:kappa}
%  \kappa( \fg^{p,q}_{W,F} \,,\, \fg^{r,s}_{W,F} ) \ = \ 0
%  \quad\hbox{if}\quad (p,q)+(r,s) \not= (0,0) \,,
%\end{equation}
and is compatible with the Lie bracket in the sense that 
\begin{equation}\label{E:lieb}
  [ \fg^{p,q}_{W,F} \,,\, \fg^{r,s}_{W,F} ] \ \subset \ \fg^{p+r,q+s}_{W,F} \,.
\end{equation}
\end{subequations}
Equation \eqref{E:NVpq} implies
\begin{equation}\label{E:Nin}
  N \ \in \ \fg^{-1,-1}_{W,F} \,.
\end{equation}

%----------------------------------------------------------
\subsubsection{}
%----------------------------------------------------------

The Lie algebra of the parabolic subgroup $P_{W,\bC} \subset G_\bC$ preserving the weight filtration $W$ is 
\begin{equation}\label{E:pW}
  \fp_W(\bC) \ = \ W_0(\fg_\bC) \ = \ 
  \bigoplus_{p+q\le0} \fg^{p,q}_{W,F} \,.
\end{equation}
Likewise 
\[
  \ff \ = \ F^0(\fg_\bC) \ = \ \bigoplus_{p\ge0}\,\fg^{p,q}_{W,F}
\]
is the parabolic Lie algebra of the stabilizer $\tStab_{G_\bC}(F)$.

%----------------------------------------------------------
\subsection{Schubert cells and period matrix representations} \label{S:schubert}
%----------------------------------------------------------

We have $\fg_\bC = \ff \op \ff^\perp$ with 
\begin{equation}\label{E:ffperp}
  \ff^\perp \ = \ \bigoplus_{p<0}\,\fg^{p,q}_{W,F}
\end{equation}
a nilpotent subalgebra of $\fg_\bC$.  The map $\tilde g$ of \eqref{E:tilde-g} is determined by specifying that it take value in $\exp(\ff^\perp) \subset G_\bC$: 
\begin{equation}\label{E:tilde-g2}
  \tilde g : \olU \ \to \ \exp(\ff^\perp)
\end{equation}
Since \eqref{E:Nin} implies that the $N_i \in \ff^\perp$, we see that 
\begin{equation}\label{E:xi}
  \xi(t,w) \ = \ \exp( \tsum \ell(t_i) N_i ) \,\tilde g(t,w)
\end{equation}
takes value in $\exp(\ff^\perp)$.  It follows that the local lift $\tPhi(t,w) = \xi(t,w) \cdot F$ takes value in the open Schubert cell $\sS \subset \check D$
\begin{equation} \label{E:Sp}
  \sS \ = \ \exp(\ff^\perp) \cdot F 
  \ = \ \left\{ E \in \check D \ | \ 
  \tdim\,(E^a \cap \overline{F_\infty^b}) 
  = \tdim\,(F^a \cap \overline{F_\infty^b}) \,,\ 
  \forall \ a ,b \right\} \,,
\end{equation}
defined by
\[
  \overline{F_\infty^b} \ = \ 
  \bigoplus_{c \le n-b} V^{c,a}_{W,F} \,.
\]

The map $\ff^\perp \to \sS$ sending $x \mapsto \exp(x) \cdot F$ is a biholomorphism.  Since $\ff^\perp \simeq \bC^d$, with $d = \tdim\,\check D$, this biholomorphism defines a coordinate chart on $\sS$.  The associated coordinate representation of $\xi(t,w)$ is the \emph{period matrix representation of $\Phi|_\sU$}.

%----------------------------------------------------------
\subsection{Infinitesimal period relation} \label{S:not-horiz}
%----------------------------------------------------------

The local lift of the period map is subject to a differential constraint, the \emph{infinitesimal period relation} (IPR): the map \eqref{E:xi} satisfies 
\[
  (\xi^{-1} \td \xi)^{p,q} \ = \ 0 \,,\quad\forall \ p \le -2 \,,
\]
with $(\xi^{-1} \td \xi)^{p,q}$ the component of the $\ff^\perp$--valued $\xi^{-1}\td \xi$ taking value in $\fg^{p,q}_{W,F}$.  Along $Z_I^* \cap \olU$ this forces the map $\tilde g$ of \eqref{E:tPhi}, \eqref{E:tilde-g} and \eqref{E:tilde-g2} to take value in the centralizer of $\s_I$,
\begin{equation}\label{E:tilde-g3}
  \tilde g(0,w) \ \in \ \exp(\ff^\perp) \,\cap\, C_{I,\bC} \,,
\end{equation}
and to satisfy the differential constraint
\begin{equation}\label{E:tilde-g-horiz}
  (\tilde g^{-1} \td \tilde g)^{p,q} 
  \ = \ 0 \,,\quad\forall \ p \le -2 \,.
\end{equation}

%----------------------------------------------------------
\section{Construction of $\PhiSBB : \olB \to \olP$}
%----------------------------------------------------------

%----------------------------------------------------------
\subsection{Extension to proper maps} \label{S:PhiW}
%----------------------------------------------------------

The period map $\Phi_I : Z_I^* \to \Gamma_I\bs D_I$ may not be proper, however it has a proper extension \cite{MR0259958}.  The map will extend to $Z_J^* \subset Z_I$ if and only if $W(\s_J) = W(\s_I)$.  Let $\Pi$ be the finest possible partition of the power set $\{I\}$ of $\{1,\ldots,\nu\}$ that satisfies the following property: given $\pi \in \Pi$ and $I \in \pi$, if $I \subset J$ and $W(\s_I) = W(\s_J)$ then $J \in \pi$.  The $\Gamma$--conjugacy class $[W]$ of the weight filtration is well-defined along
\begin{equation}\label{E:Zpi}
  Z_\pi \ = \ \bigcup_{I \in \pi} Z_I^* \,.
\end{equation}
The intersection $Z_I \cap Z_\pi$ is the \emph{weight-closure} of $Z_I^*$.   The map $\Phi_I$ extends to the weight-closure as follows (Lemma \ref{L:PhiI}).  Given $I,J\in\pi$ with $I \subset J$ we have $Z_J^* \subset Z_I \cap Z_\pi$ and $\Gamma_J \subset \Gamma_I$.  It is also the case that $D_J \subset D_I$ (Remark \ref{R:DJinDI}).  This induces maps $\Gamma_J \bs D_J \to \Gamma_I \bs D_I$.  The composition
\[
\begin{tikzcd}
  Z_J^* \arrow[r,"\Phi_J"] 
  & \Gamma_J\bs D_J \arrow[r] & \Gamma_I \bs D_I \,.
\end{tikzcd}
\]
defines a proper extension
\begin{equation}\label{E:PhiIe}
\begin{tikzcd}
  Z_I \,\cap\, Z_\pi 
  \arrow[r,"\Phi_I"] 
  & \Gamma_I \bs D_I 
\end{tikzcd}
\end{equation}
of $\Phi_I$ to the weight closure $Z_I \cap Z_\pi$ (Lemma \ref{L:PhiI}).  By the proper mapping theorem, the image 
\[
  \wp_I \ = \ \Phi_I(Z_I \cap Z_\pi)
\]
is complex analytic space (and quasi-projective by \cite{MR4557401}).

The following lemma will be useful.  Let $\Gamma_I = C_{I,\bQ} \cap \Gamma$ be the monodromy operators centralizing the cone $\s_I$.

\begin{lemma} \label{L:nbd-Z_I}
Suppose $I \in \pi$.  There exists a neighborhood $Y \subset \olB$ of $Z_I$ so that the restriction of the period map $\Phi$ to $U = B \cap Y$ lifts to $\Gamma_I \bs D$: there is a holomorphic map $\Phi_U : U \to \Gamma_I\bs D$ so that the diagram 
\[
\begin{tikzcd}
  & \Gamma_I \bs D \arrow[d] \\
  U \arrow[r,"\Phi"'] \arrow[ru,"\Phi_U"]
  & \Gamma \bs D
\end{tikzcd}
\]
commutes.
\end{lemma}

Informally, we say \emph{the monodromy near $Z_I$ takes value in $\Gamma_I$}.

\begin{proof}
It suffices to exhibit the neighborhood $Y$, and flat sections $\mathbf{N}_i$ of $\fgl(\mathbb{V})$ over $U = B \cap Y$, with the property that $\mathbf{N}_i(u)$ is identified with $N_i$, under parallel transport, for every $u \in U$ and $i \in I$.

Given any point $a \in Z_I$, fix a coordinate neighborhood $\olU_a \subset \olB$ centered at $a$, as in \S\ref{S:vhs}.  Given any point $(t,w) \in \sU_a = B \cap \olU \simeq (\Delta^*)^k \times \Delta^r$, consider the loop $\gamma_i(s) = (t_1,\ldots,t_{i-1},t_i\,e^{2\pi\bi s}, t_{i+1},\ldots,t_k;w)$, $0 \le s \le 1$, in $\sU$.  (Here $i \in I \subset \{1,\ldots,k\}$.)  Parallel transportation around this loop defines a quasi-unipotent operator $\mathbf{T}_i(t,w) \in \tGL(\mathbb{V}_{(t,w)})$; let $\mathbf{N}_i(t,w) \in \fgl(\mathbb{V}_{(t,w)})$ denote associated nilpotent operator, cf.~\S\ref{S:not-cone}.  

Define $Y = \cup_{a \in A}\,\olU_a$.  The section $\mathbf{N}_i$ of $\fgl(\left.\mathbb{V}\right|_{\sU_a})$ is independent of our choice of coordinates, and so defines the desired section of $\fgl(\mathbb{V})$ over $U = B \cap Y$.
\end{proof}

%----------------------------------------------------------
\subsection{A topological completion} \label{S:set-c}
%----------------------------------------------------------

Consider the disjoint union
\[
  \tilde\wp_\pi \ = \ \bigsqcup_{I \in \pi} \wp_I \,.
\]
If $b \in Z_I \cap Z_J \cap Z_\pi$ then we identify the two points $x = \Phi_I(b) \in \wp_I$ and $y = \Phi_J(b) \in \wp_J$.  Let $\wp_\pi$ be the quotient of $\tilde\wp_\pi$ by the equivalence relation generated by this identification.  Set
\[
  \olP \ = \ \bigcup \wp_\pi  \,.
\]

Define a map
\begin{equation}\label{iE:Phi1Phi0}
  \PhiSBB : \olB \ \to \  \olP 
\end{equation}
by specifying $\left.\PhiSBB\right|_{Z_I \cap Z_\pi} = \Phi_I$ (modulo identifications imposed by the equivalence relation defining $\wp_\pi$).  Fix a Riemannian metric on $\olB$.  Since the fibres of $\PhiSBB$ are compact (because the maps $\Phi_I : Z_I \cap Z_\pi \to \Gamma_I\bs D_I$ are proper), there is an induced metric on $\olP$.  Endow $\olP$ with the metric topology.  

\begin{theorem}\label{T:top}
The topology on $\olP$ is Hausdorff.  The induced subspace topology on the image $\wp_I \to \olP$ coincides with the natural topology on the image.  The map $\PhiSBB : \olB \to \olP$ is continuous and proper, and $\olP$ is a topological compactification of $\wp$.
\end{theorem}

\begin{proof}
It is clear that the induced subspace topology coincides with the natural topology on the image $\wp_I = \Phi_I(Z_I \cap Z_\pi)$.  The topology on $\olP$ is Hausdorff if and only if the map $\PhiSBB$ is continuous.  In this case, the map is necessarily proper.  So it suffices to establish the continuity of $\PhiSBB$.  

Suppose that $b_i \in \olB$ is a sequence of points converging to $b_\infty \in \olB$.  Let $\fibre_i$ and $\fibre_\infty$ be the fibres of $\PhiSBB$ through $b_i$ and $b_\infty$, respectively.  Let $b_i'\in\fibre_i$.  Since $\olB$ is compact, $\{b_i'\}$ contains a convergent subsequence; abusing notation, let $\{b_i'\}$ denote that convergent subsequence with limit $b'_\infty$.  The essential point is to prove that
\begin{equation}\label{E:top}
  b'_\infty \ = \ \lim_{i \to \infty} b'_i \ \in \ \fibre_\infty \,.
\end{equation}
Informally this says
\[
  \lim_{i \to \infty} A_i \ \subset \ A_\infty \,.
\]
The local analog of this assertion is Lemma \ref{L:rwfp}.  The ``globalization'' will follow from a certain finiteness result for Siegel domains.

Fix two coordinate charts $\olU$ and $\olU'$ centered at $b_\infty$ and $b_\infty'$ respectively, and local lifts $\tPhi(t,w)$ and $\tPhi'(t,w)$, as in \S\ref{S:not-loclift}.  

\smallskip

\emph{Case 1.}  First assume that both sequences $\{b_i\}$ and $\{b_i'\}$ are contained in $B$.  Without loss of generality, $b_i \in \sU$ and $b'_i \in \sU'$.  Since $b_i,b_i' \in \fibre_i$, there exists $\gamma_i \in \Gamma$ so that $\tPhi'(b'_i) = \gamma_i \cdot \tPhi(b_i)$.  Shrinking $\olU$ if necessary, there exists a finite union $\fD \subset D$ of Siegel sets so that $\tPhi(\widetilde \sU) \subset \fD$.  (In the case of one-variable degenerations this is a corollary of Schmid's $\tSL(2)$ orbit theorem \cite[(5.26)]{MR0382272}.  In the general case, this is \cite[Theorem 1.5]{MR4155216}.)  Likewise, we have a finite union $\fD' \subset D$ of Siegel sets so that $\tPhi'(\widetilde \sU') \subset \fD'$.  It follows that there are only finitely many distinct $\gamma_i$.  Restricting to a subsequence with all $\gamma_i = \gamma$ equal, we have $\tPhi'(b'_i) = \gamma \cdot \tPhi(b_i)$.  Since we may replace the local lift $\tPhi'$ with $\gamma^{-1}\cdot\tPhi'$, we may assume without loss of generality that $\tPhi'(b'_i) = \tPhi(b_i)$.  This forces $b_\infty$ and $b_\infty'$ to lie in the same $\PhiSBB$--fibre, and establishes the desired \eqref{E:top}, in the case that $\{b_i,b_i'\} \subset B$.

\smallskip

We now turn to the general case.  Since $\PhiSBB(b_i) = \PhiSBB(b_i')$, it follows that $b_i$ and $b_i'$ lie in the same weight strata $Z_{\pi_i}$, for some $\pi_i \in \Pi$.  Since $\Pi$ is finite, we may assume without loss of generality that the $\{ b_i , b_i'\} \subset Z_\pi$, for some fixed $\pi \in \Pi$.  Let $\pi_\mathrm{min} = \{ I \in \pi \ | \ J \in \pi \hbox{ and } J \subset I \hbox{ implies } J = I \}$ be the set of minimal elements (with respect to containment) in $\pi$.  Note that 
\[
   Z_\pi \ = \ \bigcup_{I \in \pi_\mathrm{min}} 
   Z_I \cap Z_\pi \,,
\]
and the $\{ Z_I \cap Z_\pi \ | \ I \in \pi_\mathrm{min}\}$ are the irreducible components of $Z_\pi$.  As there are only finitely many such components, we may assume without loss of generality that there exists $I,I' \in \pi_\mathrm{min}$ so that $\{b_i\} \subset Z_I \cap Z_\pi$ and $\{b_i'\} \subset Z_{I'} \cap Z_\pi$.

\smallskip

\emph{Case 2.}  Suppose that $I = I'$ and $\Phi_I(b_i) = \Phi_I(b_i')$.  Let $\tgr(F_I)$ be the local lift of $\Phi_I$ in $Z_I^* \cap \olU$, and $\tgr(F_I)'$ be the local lift of $\Phi_I$ in $Z_I^* \cap \olU{}'$.  Then Lemma \ref{L:nbd-Z_I} ensures $\tgr(F_I)'(b_i') = \gamma_i \cdot \tgr(F_I)(b_i)$ for some $\gamma_i \in \Gamma_I = C_{I,\bQ} \cap \Gamma$.  Now the argument above (in Case 1) applies.  The finiteness property of Siegel sets implies that there are only finitely many distinct $\gamma_i$.  (More precisely, the action of $\Gamma_I$ on $D_I$ factors through the quotient $p_I : \Gamma_I \to \Gamma_I/(\Gamma \cap C_{I,\bQ}^{-1})$, and there are only finitely many $p_I(\gamma_i)$.)  So we may normalize the local lifts so that $\tgr(F_I)'(b_i') = \tgr(F_I)(b_i)$.  Then the desired \eqref{E:top} follows from Lemma \ref{L:rwfp}.

\smallskip

\emph{Case 3.}  Now suppose that either $I = I'$ but $\Phi_I(b_i) \not= \Phi_I(b_i')$, or that $I \not= I'$.  Let $\sA_i \subset Z_I \cap Z_\pi$ be the $\Phi_I$--fibre through $b_i$, and let $\sA_i' \subset Z_{I'} \cap Z_\pi$ be the $\Phi_{I'}$--fibre through $b_i'$.  Then $\PhiSBB(b_i) = \PhiSBB(b_i')$, and the equivalence relation defining $\wp_\pi = \tilde\wp_\pi/\sim$, implies that there exists some $\Phi_{J_i}$--fibre $\sB_i$, with $J_i \in \pi_\mathrm{min}$ so that $\sB_i$ intersects both $\sA_i$ and $\sA_i'$.  Again, since $\pi_\mathrm{min}$ is finite, we may assume without loss of generality that the $J_i = J$ all coincide.  So $\sA_i \cap \sB_i \subset Z_I \cap Z_J \cap Z_\pi = Z_{I \cup J} \cap Z_\pi$ (in particular, $I \cup J \in \pi$, cf.~Corollary \ref{C:IW}) and $\sA_i' \cap \sB_i \subset Z_{I'} \cap Z_J \cap Z_\pi = Z_{I' \cup J} \cap Z_\pi$.  So the local lifts of $\Phi_I$ and $\Phi_{I'}$ satisfy $\tgr(F_I)(b_i) \in D_{I \cup J}$ and $\tgr(F_{I'})(b_i') \in D_{I'\cup J}$; in particular, both $\tgr(F_I)(b_i) , \tgr(F_{I'})(b_i') \in D_J$, cf.~Remark \ref{R:DJinDI}.  The fact that $\sB_i$ intersects both $\sA_i$ and $\sA_i'$ implies that $\tgr(F_{I'})(b_i') = \gamma_i \cdot \tgr(F_I)(b_i)$ for some $\gamma_i \in \Gamma_J$.  Now the arguments of Case 2 apply to yield the desired \eqref{E:top}.
\end{proof}

%----------------------------------------------------------
\subsection{A ``Stein factorization'' of $\PhiSBB$} \label{iS:gglr}
%----------------------------------------------------------

Since the period map $\Phi: B \to \Gamma\bs D$ is proper, we may consider the Stein factorization
\begin{subequations}\label{SE:stein}
\begin{equation}
\begin{tikzcd}
  B \arrow[r,"\Phi'"'] \arrow[rr,bend left,"\Phi"] 
  & \wp' \arrow[r] & \wp
  \,.
\end{tikzcd}
\end{equation}
The fibres of $\Phi'$ are connected, the fibres of $\wp' \to \wp$ are finite, and $\wp'$ is a normal complex analytic space.  Likewise, we have Stein factorizations
\begin{equation}\label{E:steinI}
\begin{tikzcd}[row sep=tiny]
  Z_I \,\cap\, Z_\pi \arrow[r,"\Phi_I'"']  
  \arrow[rr,bend left,"\Phi_I"] 
  & \wp_I' \arrow[r] & \wp_I
\end{tikzcd}
\end{equation}
\end{subequations}
of the maps \eqref{E:PhiIe}.  Again, the fibres of $\Phi'_I$ are connected, the fibres of $\wp'_I \to \wp_I$ are finite, and $\wp'_I$ is normal.

Consider the disjoint union
\[
  \tilde\wp'_\pi \ = \ \bigcup_{I \in \pi} \wp'_I \,.
\]
If $b \in Z_I \cap Z_J \cap Z_\pi$ then we identify the two points $x = \Phi_I'(b) \in \wp_I'$ and $y = \Phi_J'(b) \in \wp_J'$.  Let $\wp_\pi'$ be the quotient of $\tilde\wp_\pi'$ by the equivalence relation generated by this identification.  Set
\[
  \olP' \ = \ \bigcup \wp_\pi'  \,.
\]
Note that we have a well-defined map $\olP' \to \olP$.  The ``Stein factorization'' 
\begin{equation}\label{iE:hatP}
\begin{tikzcd}
  \olB \arrow[rr,bend left,"\PhiSBB"] \arrow[r,"\PhiSBB'"']
  & \olP' \arrow[r]
  & \olP
\end{tikzcd}\end{equation}
of the map $\PhiSBB$ in \eqref{iE:Phi1Phi0} is given by specifying that the restrictions to $Z_I \cap Z_\pi$ coincide with \eqref{SE:stein}.  Again, the fibres of $\PhiSBB'$ are connected, and the fibres of $\olP' \to \olP$ are finite.  

Since the fibres of $\PhiSBB'$ are compact, the Riemannian metric on $\olB$ defines a metric on $\olP'$.  Endow $\olP'$ with the induced metric topology.  The analog of Theorem \ref{T:top} holds by essentially the same argument.

\begin{theorem}\label{T:hattop}
The topology on $\olP'$ is Hausdorff.  The induced subspace topology on the image $\Phi_I'(Z_I \cap Z_\pi) \inj \olP'$ coincides with the natural topology on the image as a normal complex analytic space.  The maps of \eqref{iE:hatP} are continuous and proper, and $\olP'$ is a topological compactification of $\wp'$.
\end{theorem}

\begin{corollary} \label{C:top}
Let $A \subset \olB$ be a fibre of $\PhiSBB'$.  \emph{(Equivalently, $A$ is a connected component of a $\PhiSBB$--fibre.)}  Fix a neighborhood $\PhiSBB'(A)\in \sX \subset \olP'$.  Then $X = (\PhiSBB')^{-1}(\sX) \subset \olB$ is a neighborhood of $A$ with the property that the maps $\left.\PhiSBB\right|_X$ and $\PhiSBB'\big|_X$ are proper.
\end{corollary}

\begin{lemma} \label{L:top}
Let $\sK_\pi \subset \wp_\pi'$ be a compact set, and let $Y \subset \olB$ be any open set containing the compact preimage $K_\pi=(\PhiSBB')^{-1}(\sK_\pi) \subset Z_\pi$ of $\sK_\pi$ under $\PhiSBB'$.  There exists an open neighborhood $\sX$ of $\sK_\pi$ in $\olP'$ with the property that $X = (\PhiSBB')^{-1}(\sX) \subset Y$.
\end{lemma}

\begin{remark}
We will be particularly interested in applying the lemma in the case that $\sK_\pi$ is a point, and $K_\pi$ is a fibre of $\PhiSBB'$.
\end{remark}

\begin{proof}
We argue by contradiction.  Suppose that $\{x_n\} \subset \olP$ is sequence converging to $x_\infty \in \sK_\pi$, so that $A_n = (\PhiSBB')^{-1}(x_n) \not\subset Y$.  Fix $a_n \in A_n - Y$.  Since $\olB$ is compact, $\{a_n\}$ contains a convergent subsequence $a_{n_j} \to a_\infty \in \olB \bs Y$.  Since $K_\pi=(\PhiSBB')^{-1}(\sK_\pi)$ is compact (Theorem \ref{T:hattop}), there exists $\e>0$ so that $\{ b \in \olB \ | \ \mathrm{dist}(b,K_\pi) < \epsilon\} \subset Y$. In particular, $\mathrm{dist}(a_\infty,K_\pi) \ge \e$.  On the other hand, $x_n\to x_\infty$, and the continuity of $\PhiSBB'$ implies $a_\infty \in K_\pi$, a contradiction.
\end{proof}

%----------------------------------------------------------
\section{Monodromy near a $\PhiSBB'$--fibre} \label{S:fibre-nbd}
%----------------------------------------------------------

The main result of this section is Lemma \ref{L:Gamma-infty}, which imposes strong constraints on the monodromy about a $\PhiSBB'$--fibre $A$.  These constraints play a crucial role in the descent of (a multiple of) the line bundle $\Le$ (Corollary \ref{C:descends}).

%----------------------------------------------------------
\subsection{Reduced limit period map}\label{S:Phi-infty}
%----------------------------------------------------------

This section describes an important relationship between the period map $\Phi_I : Z_I^* \to \Gamma_I\bs D_I$ and the topological boundary $\partial D$ of the period domain in the compact dual $\check D$.  In general, the limit Hodge filtration $F$ associated with a point $b \in Z_I^*$ (\cf\S\ref{S:not-loclift}) will not lie in the boundary.  However, there is a ``na\"ive'', or \emph{reduced limit} $F_\infty(b)$, that does lie in $\partial D$ (\S\ref{S:RLP}).  The limit takes value in a $C_{I,\bR}$--orbit $\cO_I \subset \partial D$, and there is an induced map (defined in \S\ref{S:RLP})
\begin{equation}\label{E:PinftyI}
  \Phi^\infty_I : Z_I^* \ \to \ \Gamma_I\bs \cO_I \,.
\end{equation}

\begin{proposition} \label{P:RvE}
The period map $\Phi_I:Z_I^* \to \Gamma_I \bs D_I$ factors through the reduced limit period map $\Phi_I^\infty: Z_I^* \to \Gamma_I\bs \cO_I$.  Moreover, the map $\Phi^\infty_I$ is locally constant on $\Phi_I$--fibres.
\end{proposition}

\noindent
The proposition is proved in \S\S\ref{S:prfpmrl}--\ref{S:primeh}.

% --------------------------------------------------------
\subsubsection{Definition} \label{S:RLP}
% --------------------------------------------------------

Fix local coordinates $(t,w)$, as in \S\ref{S:not-U}, centered at a point $b \in Z_I^*$.  Let $\tPhi(z,w)$ denote the local lift \eqref{E:tPhi} of $\Phi$.  The \emph{reduced limit period} is
\begin{eqnarray}  \label{E:Finfty}
  F_\infty(w) 
  & = & 
  \lim_{\tIm\,z_j \to \infty} 
  	\tPhi(z,w)
  \\ \nonumber
  & = & \lim_{y \to \infty} \exp(\bi y N) \, \tilde g(0,w) \cdot F
  \quad \in \quad \overline D \,.
\end{eqnarray}
The first limit is taken over all $1 \le j \le k$, with $\tRe\,z_j$ bounded; the second limit is independent of the choice of $N \in \s$, \cite{MR3115136, MR3331177, MR3751295}.  Since $\tilde g(0,w)$ takes value in $C_{I,\bC}$ (\S\ref{S:not-horiz}), we see that 
\begin{equation}\label{E:Fi(w)}
  F_\infty(w) \ = \ \tilde g(0,w) \cdot F_\infty(0) \,.
\end{equation}
In particular, the map $F_\infty : \{0\}\times\Delta^r \to \check D$ is holomorphic, and takes value in the $C_{I,\bC}$--orbit of $F_\infty(0)$.  What is less obvious is that:  
(i) 
the holomorphic $F_\infty (w)$ takes value in the real orbit 
\[
  \cO_I \ = \ C_{I,\bR} \cdot F_\infty(0) \ \subset \ \check D\,.
\]
(ii)
The real orbit $\cO_I$ is open in the (complex) orbit $C_{I,\bC} \cdot F_\infty(0)$, and so is a complex submanifold of $\check D$.  % (In fact, $\cO_I$ is a CR--submanifold of the real orbit $G_\bR \cdot F_\infty(0) \subset \partial D$).
See \cite{MR3444997, MR3331177} for details.  The reduced limit $F_\infty$ is also independent of the local coordinates $(t,w)$ expressing $\tPhi$, and so induces a well-defined map \eqref{E:PinftyI}: 
\begin{equation}\label{E:FinfI}
  \Phi^\infty_I(0,w) \ = \ F_\infty(w) 
  \ = \ \tilde g(0,w) \cdot F_\infty(0)
  \quad\hbox{modulo}\quad \Gamma_I \,.
\end{equation}

%----------------------------------------------------------
\subsubsection{Proof: period map factors through reduced limit} \label{S:prfpmrl}
%----------------------------------------------------------

The two filtrations $F$ and $F_\infty(0)$ are related by the Deligne splitting of \S\ref{S:ds1} \cite[Appendix to Lecture 10]{MR3115136}
\begin{equation}\label{E:FvFinfty}
  F^p \ = \ \bigoplus_{a\ge p} V^{a,b}_{W,F} \tand
  F^p_\infty(0) \ = \ \bigoplus_{b \le \sfn-p} V^{a,b}_{W,F} \,.
\end{equation}
Keeping \eqref{E:ds-conj} in mind, we see that $\tStab_{G_I}(F) \equiv \tStab_{C_{I,\bR}}(F_\infty(0))$ modulo $C^{-1}_{I,\bR}$.  So there is a natural identification 
\[
  D_I \ \simeq \ C^{-1}_{I,\bR} \backslash \cO_I \,.
\]
This identification induces
\begin{equation}\label{E:piJ}
  \pi_I : \Gamma_I\backslash \cO_I \ \to \ \Gamma_I\backslash D_I \,.
\end{equation}
Then the local coordinate expression \eqref{E:FI} for $\Phi_I$, and the local coordinate expression \eqref{E:FinfI} for $\Phi^\infty_I$, yield
\begin{equation}\label{E:factor0I}
  \Phi_I \ = \ \pi_I \circ \Phi^\infty_I \,.
\end{equation}

\begin{remark}
When $D$ is hermitian the map \eqref{E:piJ} is an isomorphism and $\Phi_I = \Phi^\infty_I$.
\end{remark}

% --------------------------------------------------------
\subsubsection{Proof of finiteness: formulation of the argument}\label{S:g=fes}
% --------------------------------------------------------

It is enough to show that $F_\infty(w)$ is constant along the $\Phi_I$--fibres in $\{0\} \times \Delta^r$.  This constancy is a consequence of the infinitesimal period relation (\S\ref{S:not-horiz}).

Recall that $\tilde g(t,w)$ takes value in $\exp(\ff^\perp)$, and $\tilde g(0,w)$ takes value in $\exp(\fc_{I,\bC})$, cf.~\eqref{E:tilde-g2} and \eqref{E:tilde-g3}.  The condition \eqref{E:Nin} implies that $\fc_I \subset \fp_W$ inherits the Deligne splitting \eqref{E:pW}:
\[
  \fc_{I,\bC} \ = \ \bigoplus_{p+q \le 0} \fc_{I,F}^{p,q} \,,\qquad
  \fc_{I,F}^{p,q} \ = \ \fc_{I,\bC} \,\cap\, \fg^{p,q}_{W,F} \,.
\]  
Then we have
\[
  \ff^\perp \,\cap\, \fc_{I,\bC} \ = \ 
  \bigoplus_{\mystack{p<0}{p+q\le0}} \fc_{I,F}^{p,q} \,.
\]  
From \eqref{SE:gpq} and \eqref{E:FvFinfty} we see that the Lie algebra $\ff_\infty$ of the stabilizer $\tStab_{G_\bC}(F_\infty(0))$ is
\begin{equation}\label{E:finfty}
  \ff_\infty \ = \ \bigoplus_{q\le0} \fg^{p,q}_{W,F} \,,
\end{equation}
so that
\[
   \ff^\perp \cap \fc_{I,\bC} \cap \ff_\infty \ = \ 
   \bigoplus_{\mystack{p<0}{q\le0}} \fc_{I,F}^{p,q} \,.
\]
Consider the decomposition 
\[
  \ff^\perp \,\cap\, \fc_{I,\bC} \ = \ \fd \ \op \ \fe \ \op \ 
  (\ff^\perp \cap \fc_{I,\bC} \cap \ff_\infty )
\]
defined by 
\[
  \fd \ = \ \bigoplus_{\mystack{p<0}{p+q=0}} \fc_{I,F}^{p,q}
  \tand
  \fe \ = \ 
  \bigoplus_{\mystack{p<0<q}{p+q<0}} \fc_{I,F}^{p,q} \,.
\]
From \eqref{E:lieb} we see that each of these three summands is a Lie subalgebra of $\ff^\perp\cap\fc_{I,\bC}$.

Since $\ff^\perp\cap\fc_{I,\bC}$ is nilpotent, the function $\tilde g(0,w)$ may be uniquely decomposed as 
\[
  \tilde g(0,w) \ = \ e(w) f(w) s(w)
\]
with $f(w) \in \exp(\fd)$, $e(w) \in \exp(\fe)$ and $s(w) \in \exp(\ff^\perp \cap \fc_{I,\bC} \cap \ff_\infty)$.  Since $\tilde g(0,w) = e(w)\,f(w) s(w) f(w)^{-1} \, f(w)$, and both $e(w)$ and $f(w) s(w) f(w)^{-1}$ take value in the unipotent radical $C_{I,\bC}^{-1}$, we may 
\begin{center}
  identify $\Phi_I(0,w)$ with $f(w)$.
\end{center}
Furthermore, since $\ff_\infty$ is the stabilizer of $F_\infty(0)$ in $\ff$, \eqref{E:Fi(w)} implies we may 
\begin{center}
  identify $F_\infty(w)$ with $e(w)f(w)$.
\end{center}
So to prove the lemma, it suffices to show that
\begin{center}
  $e(w)$ is locally constant along $f$--fibres.
\end{center}
So we assume
\begin{equation} \label{E:de}
  \td f = 0 \,,
\end{equation}
and will show that $\td e = 0$; equivalently,
\begin{equation} \label{E:df}
  e^{-1} \td e \ = \ 0 \,.
\end{equation}

% --------------------------------------------------------
\subsubsection{Proof of finiteness: the IPR}\label{S:primeh}
% --------------------------------------------------------

At $(0,w)$ we have
\begin{eqnarray}
  \nonumber
  \tilde g^{-1} \td \tilde g & = & (efs)^{-1} \td (efs) \\
  \label{E:mcg}
  & = & \tAd_{fs}^{-1}( e^{-1} \td e ) \ + \ 
  \tAd_s^{-1}( f^{-1} \td f ) \ + \ s^{-1} \td s \\
  \nonumber
  & \stackrel{\eqref{E:de}}{=} &
  \tAd_{fs}^{-1}( e^{-1} \td e ) \ + \ s^{-1} \td s \,.
\end{eqnarray}
Note that $e^{-1}\td e$ and $s^{-1}\td s$ take value in $\fe$ and $\ff_\infty$, respectively.  Furthermore, \eqref{E:lieb} and $fs \in \exp(\ff^\perp\cap\fc_{I,\bC})$ imply that 
\[
  e^{-1}\td e \ = \ 0 
  \quad\hbox{if and only if} \quad
  \left(\tAd_{fs}^{-1}( e^{-1} \td e ) \right)^{p,q} \ = \ 0
\]
for all $q>0$ and $p+q < 0$.  At the same time \eqref{E:lieb}, \eqref{E:tilde-g-horiz} and \eqref{E:mcg} imply that
\[
  0 \ = \ (\tilde g^{-1} \td \tilde g)^{p,q} \ = \ 
  \left(\tAd_{fs}^{-1}( e^{-1} \td e ) \right)^{p,q}
\]
for all $q>0$ and $p+q<0$.  The desired \eqref{E:df} now follows, completing the proof of Proposition \ref{P:RvE}.

%----------------------------------------------------------
\subsection{Monodromy about the fibre} \label{S:monpt}
%----------------------------------------------------------

Fix a $\PhiSBB'$ fibre $A \subset \olB$; a coordinate neighborhood $\olU \subset \olB$ centered at a point $a_o \in A$; and a local lift $\tPhi : \widetilde\sU \to D$ of $\Phi|_\sU$, as in \S\ref{S:vhs}.  These choices determine a reduced period limit $F_\infty$, \cf~\eqref{E:Finfty}.  Define
\begin{equation}\label{E:mon-infty}
  \monzero \ = \ \tStab_\Gamma(F_\infty) \,.
\end{equation}

\begin{lemma} \label{L:Gamma-infty}
We may choose the neighborhood $X$ of Corollary \ref{C:top} so that the restriction of the period map $\Phi$ to $U = B \cap X$ lifts to $\Gamma_\infty \bs D$: there is a holomorphic map $\Phipt : U \to \monzero\bs D$ so that the diagram 
\begin{equation}\label{E:Pfibre0}
\begin{tikzcd}
  & \monzero \bs D \arrow[d] \\
  U \arrow[r,"\Phi"'] \arrow[ru,"\Phipt"]
  & \Gamma \bs D
\end{tikzcd}
\end{equation}
commutes.
\end{lemma}

Informally, we say \emph{the monodromy near $A$ takes value in $\monzero$}.

\begin{proof}
Given any point $a \in A$, fix a coordinate neighborhood $\olU_a \subset \olB$ centered at $a$, and a local lift  $\tPhi_a : \widetilde\sU_a \to D$ of $\Phi|_{\sU_a}$, as in \S\ref{S:vhs}.  Let $F_\infty(a)$ be the associated reduced period limit (\S\ref{S:RLP}).  Since the fibre $A$ is connected, Proposition \ref{P:RvE} and \eqref{E:FinfI} imply that the $\Gamma$--congruence class of $F_\infty(a)$ is independent of our choice of $a \in A$. So $F_\infty(a) = \gamma_a\cdot F_\infty$ for some $\gamma_a \in \Gamma$.  Replacing $\tPhi_a$ with $\gamma_a^{-1}\cdot\tPhi_a$, we may assume that $F_\infty(a) = F_\infty$ for all $a \in A$.  This determines the lift $\tPhi_a$ up to the action of $\monzero = \tStab_\Gamma(F_\infty)$.  

Define $Y = \cup_{a \in A}\,\olU_a$.  Since the local lifts are defined up to the action of $\Gamma_\infty$, we can patch the local lifts $\{\tPhi_a : \widetilde\sU_a \to D \}_{a \in A}$ together to define a map $\Psi : U = B \cap Y \ \to \ \monzero \bs D$, through which the restriction $\left.\Phi\right|_U$ factors.  By Lemma \ref{L:top}, we may choose $X \subset Y$.  Take $\Phipt = \Psi|_X$.
\end{proof}

\begin{remark}[Period matrix representation] \label{R:Sp}
The local lifts $\tPhi_a$ all take value in the open Schubert cell $\sS \subset \check D$ defined in \eqref{E:Sp}.  Since $\Gamma \subset G_\bR$, \eqref{E:mon-infty} implies
\begin{equation}\label{E:mon-infty-bar}
  \Gamma_\infty \ = \ \tStab_\Gamma(\overline{F_\infty}) \,.  
\end{equation}
It follows from the definition of $\sS$ that the action of $\Gamma_\infty$ on $\check D$ preserves $\sS$.  Consequently, the action of $\Gamma_\infty$ preserves $\sS \cap D$, and the map $\Phipt$ of \eqref{E:Pfibre0} takes value in $\monzero\bs (\sS \cap D)$.

The Schubert cell $\sS$ is biholomorphic to $\ff^\perp \simeq \bC^d$, with $d = \tdim\,D$ (\S\ref{S:schubert}).  So we may think of $\Phipt$ as giving us a (multi-valued) coordinate representation of $\left.\Phi\right|_U$ (defined up to the action of $\monzero$).  We call $\Phipt$ the \emph{period matrix representation of $\Phi|_U$}.
\end{remark}

\begin{remark} \label{R:int}
From \eqref{E:FvFinfty}, \eqref{E:mon-infty} and \eqref{E:mon-infty-bar} we see that the action of $\monzero$ on $V_\bC$ preserves the subspaces 
\begin{equation}\label{E:infty-intersection}
  F^{\sfn-p}_\infty \, \cap\, \overline{F^{\sfn-q}_\infty} \ = \ 
  \bigoplus_{\substack{a \le q \\ b \le p}}
  V^{a,b}_{W,F} \,;
\end{equation}
with the equality above a consequence of \eqref{E:ds-conj}.  Since \eqref{E:ds-WF} implies
\[
  W_\ell(V_\bC) \ = \ \sum_{p+q \le \ell} 
  (F^{\sfn-p}_\infty \, \cap\, \overline{F^{\sfn-q}_\infty}) \,,
\]
it follows that $\monzero$ preserves $W$:
\begin{equation}\label{E:monW}
  \monzero \ \subset \ P_{W,\bQ} \,.
\end{equation}
\end{remark}

%----------------------------------------------------------
\section{The descent of $\Le^{\ot\sfm}$ to $\olP'$} \label{S:descent}
%----------------------------------------------------------

Set $\sfk = \lceil (\sfn+1)/2 \rceil$.  Let $\cFe^p \to \olB$ denote Deligne's extension \cite{MR1416353} of the Hodge vector bundles $\cF^p \to B$.  Define 
\[
  \Le \ = \ \left\{
    \begin{array}{ll}
    \tdet(\cFe^\sfn) \,=\, \tdet(\cFe^1) \,,\quad & \sfn=1 \,,\\
    \tdet(\cFe^\sfn) \ot \tdet(\cFe^{\sfn-1}) \ot \cdots \ot 
    	\tdet(\cFe^\sfk) \,,\quad & 
    \sfn \hbox{ even,} \\
    \left[\tdet(\cFe^\sfn) \ot \tdet(\cFe^{\sfn-1}) 
    \ot \cdots \ot \tdet(\cFe^{\sfk+1}) \right]^{\ot2} \, 
    	\ot \tdet(\cFe^\sfk)\, \,, \  &
    \sfn \ge 3 \hbox{ odd.}
    \end{array}
  \right.
\]
Let $X \subset \olB$ the a neighborhood of the $\PhiSBB'$--fibre $A$ given by Lemma \ref{L:Gamma-infty}.

\begin{theorem} \label{T:triv}
There exists $\sfm>0$ so that the power $\Le^{\ot\sfm}$ is trivial over $X$.  If $\Gamma$ is neat, then we make take $\sfm=1$.
\end{theorem}

\begin{corollary}\label{C:descends}
There exist $\sfm>0$ so that the line bundle $\Le^{\ot\sfm}$ descends to $\olP'$.  If $\Gamma$ is neat, then we make take $\sfm=1$.
\end{corollary}

\begin{proof}[Proof of Corollary \ref{C:descends}]
Recall that $X$ is the $\PhiSBB'$--preimage of an open set $\sX \subset \olP'$ (Corollary \ref{C:top}).  It follows from Theorem \ref{T:triv} that $\Le^{\ot\sfm}$ descends to a line bundle on $\olP'$ that is trivial over $\sX$.  
\end{proof}

The remainder of \S\ref{S:descent} is devoted to the proof of Theorem \ref{T:triv}.  In outline, the argument proceeds as follows: in \S\ref{S:triv-det} the theorem is reduced to showing that a certain character $\chi_\infty : \monzero \to \bC^*$ has the property that $\chi_\infty(\gamma)$ is an $\sfm$-th root of unity, for all $\gamma \in \monzero$ (Remark \ref{R:chibeta}).  Here $\bC^* = \bC-\{0\}$ is the group of nonzero complex numbers.  We first show that $\chi_\infty(\gamma)$ has norm one, for all $\gamma \in \monzero$ (Theorem \ref{T:normone}, proved in \S\ref{S:prf-normone}), and then show that these values are $\sfm$-th roots of unity, for some $0 < \sfm \in \bZ$ (Theorem \ref{T:chi-infty}, proved in \S\ref{S:prf-rtu}).

%----------------------------------------------------------
\subsection{Introducing the character $\chi_\infty$} \label{S:triv-det}
%----------------------------------------------------------

Let $\widetilde U \to U$ be the universal cover of $U = B \cap X$, and $\sS$ the Schubert cell of Remark \ref{R:Sp}.  Then the map $\Phipt$ of \eqref{E:Pfibre0} lifts to the universal cover to yield a commutative diagram
\[
\begin{tikzcd}
  \widetilde U \arrow[r,"\tPhi_A"] \arrow[d]
  & \sS \cap D \arrow[d] \\
  U \arrow[r,"\Phipt"'] & \monzero \bs (\sS \cap D) \,.
\end{tikzcd}
\]  
The map $\exp(\ff^\perp) \to \exp(\ff^\perp) \cdot F = \sS$ is a biholomorphism (\S\ref{S:schubert}).  So there is a uniquely determined holomorphic map
\begin{subequations}\label{E:g-pmr}
\begin{equation}
  g : \widetilde U \ \to \ \exp(\ff^\perp) \ \subset \ G_\bC
\end{equation}
so that 
\begin{equation}
  \tPhi_A(\z) \ = \ g(\z) \cdot F \,.
\end{equation}
\end{subequations}

Fix a limiting mixed Hodge structure $(W,F)$ arising along $A$ as in \S\ref{S:not-LMHS}.  Let $d_p = \tdim\,F^p = \mathrm{rank}\,\cF^p$, and set
\[
  H \ = \ 
    \left\{
    \begin{array}{ll}
    \tw^{d_\sfn} V \,=\, \tw^{d_1}V \,,\quad & \sfn=1 \,,\\
    (\tw^{d_\sfn}V) \ot (\tw^{d_{\sfn-1}}V) \ot \cdots \ot 
    	(\tw^{d_\sfk}V) \,,\quad & 
    \sfn \hbox{ even,} \\
    \left[(\tw^{d_\sfn}V) \ot (\tw^{d_{\sfn-1}}V) \ot \cdots \ot 
    	(\tw^{d_{\sfk-1}}V) \right]^{\ot2} \, 
    	\ot (\tw^{d_\sfk}V)\, \,, \  &
    \sfn \ge 3 \hbox{ odd.}
    \end{array}
  \right.
\]
Each $\tw^{d_p}F^p$ is a line in $\tw^{d_p}V$, and 
\[
  L \ = \ \left\{
    \begin{array}{ll}
    \tw^{d_\sfn}F^\sfn \,=\, \tw^{d_1}F^1 \,,\quad & \sfn=1 \,,\\
    (\tw^{d_\sfn}F^\sfn) \ot (\tw^{d_{\sfn-1}}F^{\sfn-1}) 
    \ot \cdots \ot (\tw^{d_\sfk}F^\sfk) \,,\quad & 
    \sfn \hbox{ even,} \\
    \left[(\tw^{d_\sfn}F^\sfn) \ot (\tw^{d_{\sfn-1}}F^{\sfn-1})  
    \ot \cdots \ot (\tw^{d_{\sfk+1}}F^{\sfk+1}) \right]^{\ot2} \, 
    	\ot (\tw^{d_\sfk}F^\sfk)\, \,, \  &
    \sfn \ge 3 \hbox{ odd.}
    \end{array}
  \right.
\]
is a line in $H_\bC$.  Fix nonzero $\lambda \in L$ , let $\sfm$ be any positive integer and define
\[
  f : \widetilde U \to \ H_\bC^{\ot \sfm}
\]
by
\[
  f(\zeta) \ = \ g(\zeta) \cdot \lambda^\sfm \,.
\]

\begin{remark} \label{R:fgamma}
If 
\begin{equation}\label{E:fgamma}
  f(\zeta \cdot \gamma) \ = \ \gamma^{-1} \cdot f(\zeta) \,,
  \quad \forall \ \zeta \in \widetilde U \,,\ 
  \gamma \in \monzero \,,
\end{equation}
then $f$ defines a section of $\Lambda^\sfm \to U$.  This section will define a trivialization of $\Le^\sfm$ over $X$, by essentially the same arguments as in \cite{MR1416353}.  So to prove the theorem, it suffices to establish \eqref{E:fgamma} for some choice of $1 \le \sfm \in \bZ$.
\end{remark} 

%----------------------------------------------------------
\subsubsection{Decomposing monodromy} \label{S:decomps}
%----------------------------------------------------------

From \eqref{E:mon-infty} and \eqref{E:mon-infty-bar} we see that
\begin{subequations} \label{SE:monzero}
\begin{equation}
  \monzero \ \subset \ S \,,
\end{equation}
where
\begin{equation}  
  S \ = \ 
  \tStab_{G_\bC}(F_\infty) \,\cap\, 
  \tStab_{G_\bC}(\overline{F}_\infty) \,.
\end{equation}
\end{subequations}
The proof of the theorem will make use of a decomposition \begin{equation}\label{E:Sinfty}
  S \ = \ S^{-1}_\infty \,\rtimes\, S^0_\infty
\end{equation}
that is induced by a natural decomposition of the parabolic group $\tStab_{G_\bC}(\overline{F_\infty})$, \cite[Theorem 3.1.3]{MR2532439}.  Specifically 
\[
  S^{-1}_\infty \ = \ \{ g \in S \ | \ g \hbox{ acts trivially on }
  \overline F{}^q_\infty/\overline F{}^{q+1}_\infty \,,
  \ \forall \ q \} \ = \ S \,\cap\, \exp(\ff^\perp)
\]
and 
\[
   S^0_\infty \ = \ S \,\cap\,\tStab_{G_\bC}(F) \,.
\]

%----------------------------------------------------------
\subsubsection{Monodromy action}
%----------------------------------------------------------

We have 
\[
  \tPhi_A(\zeta\cdot\gamma) \ = \ 
  \gamma^{-1} \cdot \tPhi_A(\zeta) \,;
\]
equivalently, 
\[
  g(\zeta\cdot\gamma) \cdot F \ = \ \gamma^{-1} g(\z) \cdot F \,.
\]
However, while $\gamma^{-1}$ preserves the Schubert cell $\sS$, it need not be an element of $\exp(\ff^\perp)$.  So we cannot assert that $g(\zeta\cdot\gamma)=\gamma^{-1} g(\z)$.  

In order to determine $g(\zeta\cdot\gamma)$ we factor the monodromy using \eqref{E:Sinfty}.  Write
\[
  \gamma^{-1} \ = \ \a\,\b \,, 
\]
with $\beta \in S^0_\infty$ and $\a \in S^{-1}_\infty$.  Then the action of $\gamma^{-1}$ on $g \cdot F \in \sS$ is given by 
\begin{equation}\label{E:gamma-act}
  \gamma^{-1} g \cdot F \ = \ \a\b g \cdot F \ = \ 
  \a\,\b g \b^{-1}\,\b\cdot F \ = \ \a\,(\b g \b^{-1}) \cdot F \,.
\end{equation}
Noting that $\b g \b^{-1} \in \exp(\ff^\perp) = S^{-1}_\infty$, this implies that 
\[
  g(\zeta\cdot\gamma) \ = \ \a\,\b g(\zeta) \b^{-1} \,.
\]
Since $\b^{-1}$ preserves $F$, it stabilizes the line $L$.  So there is a group homomorphism 
\[
  \chi : S^0_\infty \ \to \ \bC^* \,=\, \bC\bs\{0\}
\]
such that 
\[
  \beta^{-1}(\lambda) \ = \ \chi(\beta^{-1}) \,\lambda \,.
\]

\begin{remark} \label{R:chibeta}  %%% REVISION
If every $\chi(\beta^{-1})$ is an $\sfm$--root of unity, then \eqref{E:fgamma} will hold.  We will show that there is a group homomorphism $\chi_\infty : \monzero \to S^1 \subset \ \bC^*$ such that $\overline{\chi_\infty(\gamma)} = \chi(\beta)$, and each $\chi_\infty(\gamma)$ is a root of unity (Lemma \ref{L:chi1}, and Theorems \ref{T:normone} and \ref{T:chi-infty}.  Since $\monzero$ is contained in an arithmetic group, and every arithmetic group contains a neat subgroup of finite index, it follows that there exists a choice of $\sfm$ so that the elements of $\chi_\infty(\monzero) \subset S^1$ are all $\sfm$--th roots of unity.  This will establish Theorem \ref{T:triv} (cf.~Remark \ref{R:fgamma}).
\end{remark}

%----------------------------------------------------------
\subsubsection{The character $\chi_\infty$} \label{S:chi-infty}
%----------------------------------------------------------

Because $\monzero$ stabilizes $F_\infty^p \subset V$, the line
\[
  L_\infty \ = \ \left\{
    \begin{array}{ll}
    \tw^{d_\sfn}F_\infty^\sfn \,=\, 
    \tw^{d_1}F_\infty^1 \,,\quad & \sfn=1 \,,\\
    (\tw^{d_\sfn}F_\infty^\sfn) \ot 
    (\tw^{d_{\sfn-1}}F_\infty^{\sfn-1}) 
    \ot \cdots \ot (\tw^{d_\sfk}F_\infty^\sfk) \,,\quad & 
    \sfn \hbox{ even,} \\
    \left[(\tw^{d_\sfn}F_\infty^\sfn) \ot 
    (\tw^{d_{\sfn-1}}F_\infty^{\sfn-1})  
    \ot \cdots \ot 
    (\tw^{d_{\sfk+1}}F_\infty^{\sfk+1}) \right]^{\ot2} \, 
    	\ot (\tw^{d_\sfk}F_\infty^\sfk)\, \,, \  &
    \sfn \ge 3 \hbox{ odd}
    \end{array}
  \right.
\]
is an eigenline for the action of $\monzero$ on $H_\bC$.  Let 
\[
  \chi_\infty : \monzero \ \to \ \bC^*
\]
be the associated character.  

\begin{lemma} \label{L:chi1}
We have $\chi(\beta) = \overline{\chi_\infty(\gamma)}$.
\end{lemma}

\noindent  The lemma is proved in \S\ref{S:prf-chi1}.  First we need to review the induced mixed Hodge structure on $H$.

%----------------------------------------------------------
\subsection{Induced mixed Hodge structure on $H$} \label{S:indH}
%----------------------------------------------------------

Any limiting mixed Hodge structure $(Q,W,F,N)$ on $V$ naturally induces one on $H$.  Because these induced structures play an important role in the several of the arguments that follow, we review their definition and properties here.

It will be convenient to use $H \subset V^{\ot d}$, where
\[
  d \ = \ \left\{ \begin{array}{ll}
    d_\sfn = d_1 \,,\quad & \sfn=1 \,,\\
    d_\sfn + d_{\sfn-1}+\cdots + d_\sfk \,,\quad 
    & \sfn \hbox{ even,}\\
    2\,(d_\sfn + d_{\sfn-1} + \cdots + d_{\sfk+1}) + d_\sfk \,,
    \quad & \sfn \ge 3 \hbox{ odd.}
  \end{array} \right.
\]

%----------------------------------------------------------
\subsubsection{Induced polarization}  
%----------------------------------------------------------

The polarization $Q$ of $V$ naturally induces a nondegenerate $(-1)^{\sfn d}$--symmetric bilinear form on $V^{\ot d}$, also denoted $Q$, and given by 
\begin{equation} \label{E:indQ}
  Q(u_1 \ot \cdots \ot u_d \,,\, 
  	v_1 \ot \cdots \ot v_d) \ = \ 
  Q(u_1,v_1) \cdots Q(u_d,v_d) \,, 
\end{equation}
for all $u_i,v_i \in V$.  The restriction of this bilinear form to $H \subset V^{\ot d}$ is nondegenerate.

%----------------------------------------------------------
\subsubsection{Induced mixed Hodge structure}  \label{S:ind-MHS}
%----------------------------------------------------------

A mixed Hodge structure $(W,F)$ on $V$ induces one on $V^{\ot d}$.  The weight and Hodge filtrations are given by
\begin{subequations}\label{SE:ind-MHS}
\begin{equation}\label{E:Wten}
  W_\ell(V^{\ot d}) \ = \ 
  \sum_{\ell_1+\cdots+\ell_d \le \ell}
  W_{\ell_1}(V) \ot \cdots \ot W_{\ell_d}(V)
\end{equation}
and
\begin{equation}\label{E:Ften}
  F^p(V^{\ot d}_\bC) \ = \ 
  \sum_{p_1+\cdots+ p_d \ge p}
  F^{p_1}(V_\bC) \ot \cdots \ot F^{p_d}(V_\bC) \,.
\end{equation}
The subspace $H \subset V^{\ot d}$ is naturally a mixed Hodge substructure.  Explicitly
\begin{equation}
  W_\ell(H) \ = \ H \,\cap\, W_\ell(V^{\ot d}) 
  \tand
  F^p(H_\bC) \ = \ H_\bC \,\cap\, F^p(V^{\ot d}) \,.
\end{equation}
\end{subequations}
From \eqref{SE:QW}, \eqref{E:indQ} and \eqref{SE:ind-MHS} we see that both the induced weight and Hodge filtrations are $Q$--isotropic
\begin{subequations} \label{SE:indQWF}
\begin{eqnarray}
  \label{E:indQW}
  Q(W_k(H),W_\ell(H)) & = & 0 \,,\quad 
  \forall \ k + \ell < 2 \sfw \,,\\
  \label{E:indQF}
  Q(F^p(H_\bC),F^q(H_\bC)) & = & 0 \,,\quad 
  \forall \ p+q > \sfw \,.
\end{eqnarray}
\end{subequations}

%----------------------------------------------------------
\subsubsection{Induced nilpotent operator} \label{S:indN}
%----------------------------------------------------------

The nilpotent operator $N : V \to V$ induces a nilpotent $N : V^{\ot d} \to V^{\ot d}$ by 
\begin{eqnarray*}
  N(v_1 \ot \cdots \ot v_d) & = & 
  (Nv_1) \ot v_2 \ot \cdots \ot v_d 
  \ + \ v_1 \ot (Nv_2) \ot v_3 \ot \cdots \ot v_d \\
  & & + \cdots + \ v_1 \ot \cdots v_{d-1} \ot (Nv_d) \,.
\end{eqnarray*}
This operator preserves the subspace $H \subset V^{\ot d}$.

If the mixed Hodge structure $(V,Q;W,F)$ is polarized by $N : V \to V$, then the mixed Hodge structure $(V^{\ot d},Q;W,F)$ is polarized by $N : V^{\ot d} \to V^{\ot d}$.  In particular
\begin{equation}\label{E:HN(WF)}
  N(W_\ell(H)) \ \subset \ W_{\ell-2}(H)
  \tand N(F^p(H_\bC)) \ \subset \ F^{p-1}(H_\bC) \,.
\end{equation}  
In this case, the reduced limit period is
\begin{subequations} \label{SE:Hinfty}
\begin{equation}
  F_\infty^p(V^{\ot d}) \ = \ 
  \sum_{p_1+\cdots+ p_d \ge p}
  F_\infty^{p_1}(V) \ot \cdots \ot F_\infty^{p_d}(V) \,.
\end{equation}
Likewise, the mixed Hodge substructure $(H,Q;W,F)$ is polarized by $N : H \to H$, and
\begin{equation}
  F^p_\infty(H) \ = \ H \cap F^p_\infty(V^{\ot d}) \,.
\end{equation}
\end{subequations}

If $(V,W,F)$ is a pure, weight $\sfn$, $Q$--polarized Hodge structure (that is, $W$ is trivial, and $F \in D$), then the induced $(V^{\ot d} , W , F)$ is also a pure, polarized Hodge structure of weight 
\[
  \sfw \ = \ \sfn \,d \,,
\]
and $H$ is a polarized Hodge substructure.
  
%----------------------------------------------------------
\subsubsection{Numerology}
%----------------------------------------------------------
  
Define an integer $\sfa \le \sfw$ by the conditions
\begin{equation} \label{E:a}
  F^\sfa(H_\bC) \ = \ L
  \tand 
  F^{\sfa+1}(H_\bC) \ = \ 0 \,.
\end{equation}

\begin{remark}
If $\sfn=1$ then $\sfa = d_1 = \sfw$.  Likewise, if $\sfn=2$, then $\sfa = 2 d_2 = \sfw$.  In general, $\sfa < \sfw$.  For example, if $\sfn=3$, then $\sfa = 7 d_3 + 2 d_2 \le 6 d_3 + 3 d_2 = \sfw$; the inequality if strict if and only if $d_3 < d_2$.  (In general $d_3 \le d_2$.)  Similarly, if $\sfn=4$, then $\sfa = 5 d_4 + 3 d_3 \le 4(d_4 + d_3) = \sfw$; the inequality is strict if $d_4 < d_3$.  (In general $d_4 \le d_3$.)
\end{remark}

Since $L$ is one-dimensional, \eqref{E:a} implies that there exists $\sfb \le \sfw$ so that the Deligne splitting $H_\bC = \op\,H^{p,q}_{W,F}$ satisfies
\begin{subequations}\label{SE:dsL}
\begin{equation}\label{E:dsL}
  L \ = \ F^\sfa(H_\bC) \ = \ H^{\sfa,\sfb}_{W,F} \,.
\end{equation}
Consequently,
\begin{equation}\label{E:dsLinfty}
  L_\infty \ = \ F_\infty^\sfa(H_\bC) \ = \ 
  H^{\sfw-\sfb,\sfw-\sfa}_{W,F} \,.
\end{equation}
\end{subequations}

\begin{remark} \label{R:ab}
It may be helpful to use the Hodge diamond of the mixed Hodge structure on $H$ to visualize the defining conditions of $\sfa,\sfb$, cf.~Figure \ref{fig:hd1}.  (The Hodge diamond is configuration of nodes at integer points in the $pq$-plane; a node is placed at the point $(p,q) \in \bZ^2$ if $H^{p,q}_{W,F}\not=0$.  The integers $\sfa,\sfb$ are defined so that the Hodge diamond is contained in the region $\sfw-\sfa \le p,q \le \sfa$, and there is a single node -- a single nonzero $H^{p,q}_{W,F}$ -- along each of the four boundary edges $r,s = \sfa , \sfw-\sfa$.)
\end{remark}

\begin{figure}
\begin{tikzpicture}[baseline=(current  bounding  box.center),scale=0.8]
  \draw [gray,dashed] (-0.5,6.5) --  (6.5,-0.5);
  \node [gray,above right] at (6.5,-0.5) {$p+q = \sfw$};
  \draw [gray] (0,0) -- (6,0) -- (6,6) -- (0,6) -- (0,0);
  \draw [gray] (1,1) -- (5,1) -- (5,5) -- (1,5) -- (1,1);
  \foreach \x in {1,...,5}
  {
  	\foreach \y in {1,...,5}
	{
	  	\draw [fill,gray] (\x,\y) circle [radius=0.07];
	}
  }
  \draw [fill] (3,0) circle [radius=0.08];
  \node [below left] at (3,0) {$(\sfw-\sfb,\sfw-\sfa)$};
  \node [above left] at (3,0) {$L_\infty$};
  \draw [fill] (0,3) circle [radius=0.08];
  \node [below left] at (0,3) {$(\sfw-\sfa,\sfw-\sfb)$};
  \node [above left] at (0,3) {$\overline{L_\infty}$};
  \draw [fill] (6,3) circle [radius=0.08];
  \node [below right] at (6,3) {$(\sfa,\sfb)$};
  \node [above right] at (6,3) {$L$};
  \draw [fill] (3,6) circle [radius=0.08];
  \node [below right] at (3,6) {$(\sfb,\sfa)$};
  \node [above right] at (3,6) {$\overline{L}$};
%  \node [gray] at (5.7,2.3) {$N$};
%  \draw [gray,-to] (5.8,2.8) -- (5.2,2.2);
%  \draw [gray,-to] (4.8,1.8) -- (4.2,1.2);
%  \draw [gray,-to] (3.8,0.8) -- (3.2,0.2);
\end{tikzpicture}
\label{fig:hd1}
\caption{Hodge diamond of the MHS on $H$.}
\end{figure}

%----------------------------------------------------------
\subsection{Proof of Lemma \ref{L:chi1}}  \label{S:prf-chi1}
%----------------------------------------------------------

Let $H_\bC = \op\,H^{p,q}_{W,F}$ be the Deligne splitting. The analog of \eqref{E:ds-Q} for the induced limiting mixed Hodge structure on $H$ is 
\begin{equation}\label{E:ds-indQ}
  Q\left(H^{p,q}_{W,F} \,,\, H^{r,s}_{W,F}\right) 
  \ = \ 0 \,, \quad\hbox{for all}\quad (p+r,q+s) 
  \not= (\sfw,\sfw) \,.
\end{equation}
From \eqref{SE:dsL} and \eqref{E:ds-indQ} we see that
\[
  H_\bC \ = \ 
  L^\perp \ \op \ \overline{L_\infty} 
  \ = \ \overline{L_\infty}{}^\perp \ \op \ L\,.
\]
In particular, given $0\not=\mu \in L_\infty$, the pairing $Q( \lambda \,,\, \bar\mu )$ is nonzero.  Since the polarization $Q$ is invariant under $\Gamma$, we have
\[
  Q(\gamma^{-1}\cdot\lambda \,,\, \bar\mu) 
  \ = \ Q(\lambda \,,\, \gamma\cdot\bar\mu )
  \ = \ \overline{\chi_\infty(\gamma)}\,Q(\lambda \,,\, \bar\mu) \,.
\]
On the other hand, since $\tdim\,\overline{L_\infty}=1$, the unipotent $\a \in S^{-1}_\infty$ necessarily acts trivially on $\overline{L_\infty}$, so that we have $Q(\a\cdot\lambda , \bar \mu) = Q(\lambda,\a^{-1}\cdot\bar\mu) = Q(\lambda,\bar\mu)$.  In particular, 
\[
  Q(\gamma^{-1}\cdot\lambda \,,\, \bar\mu) \ = \ 
  Q(\a\b\cdot\lambda \,,\, \bar \mu)
  \ = \ 
  \chi(\beta)\,Q(\a\cdot\lambda \,,\, \bar \mu)
  \ = \ 
  \chi(\beta)\,Q(\lambda \,,\, \bar \mu) \,.
\]
Thus $\chi(\beta)= \overline{\chi_\infty(\gamma)}$.
\hfill\qed

\begin{theorem} \label{T:normone}
The character $\chi_\infty : \monzero \to \bC^*$ takes value in the unit circle $S^1 \subset \bC^*$.
\end{theorem}

\noindent The theorem is proved in \S\ref{S:prf-normone}.

\begin{theorem} \label{T:chi-infty}  %%% NEW
Given $\gamma \in \monzero$, the eigenvalue $\chi_\infty(\gamma)$ is a root of unity.
\end{theorem}

\noindent This will establish Theorem \ref{T:triv}, cf.~Remark \ref{R:chibeta}.  Theorem \ref{T:chi-infty} is proved in \S\ref{S:prf-rtu}.

%----------------------------------------------------------
\section{Proof of Theorem \ref{T:normone}} \label{S:prf-normone}
%----------------------------------------------------------

%----------------------------------------------------------
\subsection{Decomposing $\chi_\infty$} 
%----------------------------------------------------------

Since the action of $\monzero$ on $V_\bC$ preserves the subspaces $F^p_\infty \subset V_\bC$, \cf~\eqref{E:mon-infty},  we may define a character $\chi^p_\infty : \monzero \to \bC^*$ by
\[
  \chi^p_\infty(\gamma) \ = \ \tdet\{ 
  	\gamma : F^p_\infty \ \to \ F^p_\infty \} \,.
\]
It follows from the definitions of $L_\infty$ and $\chi_\infty:\monzero \to \bC^*$, cf.~\S\ref{S:chi-infty}, that 
\begin{equation}\label{E:chi-infty}
  \chi_\infty \ = \ 
    \left\{
    \begin{array}{ll}
    \chi_\infty^\sfn \,=\, \chi_\infty^1 \,,\quad & \sfn=1 \,,\\
    \chi_\infty^\sfn\,\chi_\infty^{\sfn-1} \cdots
    	 \chi_\infty^\sfk \,,\quad 
	& \sfn = 2\sfk-2 \hbox{ even,} \\
	(\chi_\infty^\sfn\,\chi_\infty^{\sfn-1} \cdots
    	 \chi_\infty^{\sfk+1})^2\,\chi_\infty^\sfk \,, \quad  &
    \sfn = 2\sfk-1 \ge 3 \hbox{ odd.}
    \end{array}
  \right.
\end{equation}

By \eqref{E:monW}, the action of $\monzero$ on $V$ preserves the subspaces $W_\ell(V)$.  In particular, there is an induced action of $\monzero$ on $\tGr^W_\ell(V) = W_\ell(V)/W_{\ell-1}(V)$.  Since $(W,F)$ is a mixed Hodge structure, $F$ induces a Hodge decomposition 
\begin{equation} \label{E:GrW(V)pq}
  \tGr^W_\ell(V_\bC) \ = \
  \bigoplus_{p+q=\ell} \tGr^W_\ell(V)^{p,q} \,.
\end{equation}

\begin{lemma} \label{L:Gr-pq-preserved}
The action of $\monzero$ on $\tGr^W_\ell(V_\bC)$ preserves the Hodge summands $\tGr^W_\ell(V)^{p,q} \simeq V^{p,q}_{W,F}$.
\end{lemma}

\begin{proof}
This follows from \eqref{E:ds-WF} and Remark \ref{R:int}.
\end{proof}

Lemma \ref{L:Gr-pq-preserved} implies that we may define a character $\chi^{p,q} : \monzero \to \bC^*$ by
\[
  \chi^{p,q}(\gamma) \ = \ \tdet\{ 
  	\gamma : \tGr^W_\ell(V)^{p,q} \ \to \ \tGr^W_\ell(V)^{p,q} 
	\} \,.
\]
Then it follows from the second equality of \eqref{E:FvFinfty} that
\begin{equation}\label{E:chi-p}
  \chi_\infty^p \ = \ 
  \prod_{\substack{0\le a \le \sfn \\ 0 \le b \le \sfn-p}}
  \chi^{a,b} \,.
\end{equation}
Since $\overline{\tGr^W_\ell(V)^{p,q}} = \tGr^W_\ell(V)^{q,p}$, we have
\begin{subequations} \label{SE:pq-ids}
\begin{equation}
  \overline{\chi^{p,q}} \ = \ \chi^{q,p} \,.
\end{equation}
Similarly, \eqref{E:QW*}, \eqref{E:ds-WF} and \eqref{E:ds-Q} yield $(\tGr^W_\ell(V)^{p,q})^\vee \simeq \tGr^W_\ell(V)^{\sfn-p,\sfn-q}$, so that 
\begin{equation}
  (\chi^{p,q})^{-1} \ = \ \chi^{\sfn-p,\sfn-q} \,.
\end{equation}
\end{subequations}
If $p+q = \sfn$, then these two equations imply $(\chi^{p,q})^{-1} = \overline{\chi^{p,q}}$; in particular,
\begin{subequations} \label{SE:pqn-ids}
\begin{equation}
  |\chi^{p,q}| \ = \ 1 \,,\quad \forall \ p+q = \sfn \,.
\end{equation}
More generally, we have 
\begin{equation}
  |\chi^{p,q}\,\chi^{\sfn-p,\sfn-q}| \ = \ 
  |\chi^{p,q}\,\chi^{\sfn-q,\sfn-p}| \ = \ 1 \,,\quad 
  \forall \ p,q\,. 
\end{equation}
\end{subequations}

%----------------------------------------------------------
\subsection{An arithmetic observation} 
%----------------------------------------------------------

Since the weight filtration $W$ of $V$ is rational, the parabolic subgroup $P_W = \tStab_G(W)$ is a $\bQ$--algebraic group.  Let $U_W \subset P_W$ be the unipotent radical, and $L_W = P_W/U_W$ the Levi quotient.  Both $U_W$ and $L_W$ are $\bQ$--algebraic groups.  The Levi quotient is reductive, and $L_{W,\bZ} = P_{W,\bZ}/U_{W,\bZ}$ is an arithmetic subgroup, \cite[Theorem 1.2]{MR0204533}.  The group $L_W$ naturally acts on $\tGr^W_\ell(V)$, and the determinant of this action defines a character $\chi_\ell : L_W \to \bC^*$.

\begin{lemma} \label{L:chi-ell}
If $g \in L_{W,\bZ}$, then $\chi_\ell(g) \in \{ \pm 1\}$.
\end{lemma}

\begin{proof}
Since $W$ is $Q$--isotropic, \cf~\eqref{SE:QW}, we have 
\[
  \chi_{2\sfn-\ell} \ = \ \chi_\ell^{-1} \,.
\]
So it suffices to consider $\ell \ge \sfn$.  In the case that $\ell = \sfn$, we have $\chi_\sfn = \chi_\sfn^{-1}$.  So it suffices to consider the case $\ell > \sfn$.

As a $\bQ$--algebraic group, we have  
\[
  L_W \ \simeq \ \tGL(\tGr^W_\sfn(V),Q) \ \times \ 
  \prod_{\ell = \sfn+1}^{2 \sfn} \tGL(\tGr^W_\ell(V)) \,.
\]
Fix $\ell > \sfn$, and consider the projection $\pi_\ell : L_W \to \tGL(\tGr^W_\ell(V))$.  We have $\tGL(\tGr^W_\ell(V_\bQ)) \simeq \tGL_{m_\ell}(\bQ)$, where $m_\ell = \tdim\tGr^W_\ell(V)$.  The character $\chi_\ell(g)$ is the determinant of the matrix representation of $\pi_\ell(g)$, where $g \in L_W$.  The image $\pi_\ell(L_{W,\bZ})$ is an arithmetic subgroup of $\tGL(\tGr^W_\ell(V_\bQ))$, \cite[Theorem 1.2]{MR0204533}.  This implies that the image $\chi_\ell(L_{W,\bZ})$ is finite, and lies in $\bQ$.
\end{proof}

We may naturally regard $\chi_\ell$ as a character of $P_W$.  When restricted to $\monzero \subset P_{W,\bZ}$, \eqref{E:GrW(V)pq} and Lemma \ref{L:chi-ell} imply that 
\begin{equation} \label{E:chi-ell}
   \chi_\ell \ = \ \prod_{p+q = \ell} \chi^{p,q} \ : \ 
   \monzero \ \to \ \{\pm1\} \,.
\end{equation}
Simple consequences of \eqref{SE:pq-ids} and \eqref{E:chi-ell} include:
%\begin{subequations}\label{SE:chi}
\begin{eqnarray*}
  %\nonumber
  1 & = & |\chi_0| \ = \ |\chi^{0,0}| \\
  %\nonumber
  1 & = & |\chi_1| \ = \ 
  	|\chi^{1,0}\,\chi^{0,1}| \ = \ 
	|\chi^{1,0}|^2 \ = \ |\chi^{0,1}|^2\\
  %\nonumber
  1 & = & |\chi_2| \ = \ 
  	|\chi^{2,0}\,\chi^{2,2}\,\chi^{0,2}|
	\ = \ |\chi^{2,0}|^2 \,|\chi^{1,1}| \\
  %\nonumber 
  & \vdots & \\
  1 & = & 
  %\label{E:chiodd}
  |\chi_{2a-1}| \ = \ 
  | \chi^{2a-1,0} \cdots \chi^{a,a-1} |^2
  \,,\qquad \forall \ 2a-1 \le \sfn \,,\\
  %\label{E:chiev}
  1 & = & |\chi_{2a}| \ = \ 
  | (\chi^{2a,0} \cdots \chi^{a+1,a-1})^2 \,
  \chi^{a,a} | \,,\quad \forall \ 2a \le \sfn \,.
\end{eqnarray*}
%\end{subequations}
More generally, given $0 \le a \le \sfn$ and $m_{p,q} \in \bZ$ so that $m_{p,q}+m_{q,p} = 2m$ is independent of $p,q$, \eqref{SE:pq-ids} and \eqref{E:chi-ell} yield
\begin{equation}\label{E:sum}
  1 \ = \ |\chi_a|^m \ = \ 
  |\chi^{a,0}|^{m_{a,0}} \,
  |\chi^{a-1,1}|^{m_{a-1,1}} \cdots 
  |\chi^{1,a-1}|^{m_{1,a-1}} \,
  |\chi^{0,a}|^{m_{0,a}} \,.
\end{equation}

\bigskip

We are now ready to prove Theorem \ref{T:normone}: $|\chi_\infty(\gamma)| = 1$ for all $\gamma \in \monzero$.  The basic idea behind the proof -- that the theorem is a consequence of the decomposition of $\chi_\infty$ given by \eqref{E:chi-infty} and \eqref{E:chi-p}, and the identities \eqref{SE:pq-ids}, \eqref{SE:pqn-ids} and \eqref{E:sum} -- can be obscured by the intricate notation required to establish the general case (\S\ref{S:prf-ev}--\ref{S:prf-od}).  So it is instructive to first consider the simpler arguments in the cases $\sfn=1,2,3$. 

%----------------------------------------------------------
\subsection{Proof of Theorem \ref{T:normone}: weight $\sfn=1$}
%----------------------------------------------------------

According to \eqref{E:chi-infty}, we wish to show that $\chi_\infty = \chi^1_\infty$ has norm one.  From \eqref{E:chi-ell} and \eqref{SE:pq-ids} we see that
\begin{eqnarray*}
  1 & = & |\chi_0|
  \ = \ |\chi^{0,0}| \,,\\
  1 & = & |\chi_1|
  \ = \ |\chi^{1,0}\,\chi^{0,1}|
  \ = \ |\chi^{1,0}|^2 \,.
\end{eqnarray*}
Then \eqref{E:chi-p} yields $|\chi^1_\infty| = |\chi^{0,0}\,\chi^{1,0}| = 1$, as desired.  \hfill\qed

%----------------------------------------------------------
\subsection{Proof of Theorem \ref{T:normone}: weight $\sfn=2$}
%----------------------------------------------------------

According to \eqref{E:chi-infty}, we wish to show that $\chi_\infty = \chi^2_\infty$ has norm one.  From \eqref{E:chi-ell} and \eqref{SE:pq-ids} we see that 
\begin{eqnarray*}
  1 & = & |\chi_0|
  \ = \ |\chi^{0,0}| \,,\\
  1 & = & |\chi_1|
  \ = \ |\chi^{1,0}\,\chi^{0,1}|
  \ = \ |\chi^{1,0}|^2
\end{eqnarray*}
And from \eqref{SE:pqn-ids} we see that 
\[
  1 \ = \ |\chi^{2,0}|\,.
\]
Thus \eqref{E:chi-p} yields $|\chi^2_\infty| = |\chi^{0,0}\,\chi^{1,0}\,\chi^{2,0}| = 1$, as desired. \hfill\qed

%----------------------------------------------------------
\subsection{Proof of Theorem \ref{T:normone}: weight $\sfn=3$}
%----------------------------------------------------------

According to \eqref{E:chi-infty}, we wish to show that $|\chi_\infty| = | (\chi^3_\infty)^2\,\chi^2_\infty|=1$.  From \eqref{E:chi-ell} and \eqref{SE:pq-ids} we see that  
\begin{eqnarray*}
  1 & = & |\chi_0| \ = \ |\chi^{0,0}| \,,\\
  1 & = & |\chi_1|
  \ = \ |\chi^{1,0}\,\chi^{0,1}|
  \ = \ 
  	|\chi^{1,0}|^2 \ = \ |\chi^{0,1}|^2 \,, \\
  1 & = & |\chi_2|
  \ = \ 
  	|\chi^{2,0}\,\chi^{1,1}\,\chi^{2,0}|
  \ = \
  	|(\chi^{2,0})^2\,\chi^{1,1}| \,.
\end{eqnarray*}
And from \eqref{SE:pqn-ids} we see that 
\begin{eqnarray*}
  1 & = & |\chi^{3,0}| \,,\ |\chi^{2,1}| \\
  1 & = & | \chi^{3,1}\,\chi^{2,0} | \,.
\end{eqnarray*}
Keeping these identities in mind, \eqref{E:chi-p} yields
\begin{eqnarray*}
  |\chi_\infty^3| & = & 
  	|\chi^{0,0}\,\chi^{1,0}\,\chi^{2,0}\,\chi^{3,0}|
  \ = \ |\chi^{2,0}| \,,\\
  |\chi_\infty^2| & = & 
  	|\chi_\infty^3| \cdot 
  	|\chi^{0,1}\,\chi^{1,1}\,\chi^{2,1}\,\chi^{3,1}| \ = \ 
	|\chi^{2,0}| \cdot 
  	|\chi^{1,1}\,\chi^{3,1}| \,,
\end{eqnarray*}
so that 
$|(\chi_\infty^3)^2\,\chi_\infty^2| = 
|(\chi^{2,0})^2\,\chi^{1,1}| \cdot | \chi^{3,1}\,\chi^{2,0} | 
= 1$, as desired.\hfill\qed

%----------------------------------------------------------
\subsection{Proof of Theorem \ref{T:normone}: even weight} \label{S:prf-ev}
%----------------------------------------------------------

It will be helpful to keep in mind that 
\[
  \sfn \,+\, \sfk \,-\, 2 \ = \ 2 \sfn - \sfk
  \tand \sfn \,-\, \sfk \ = \ \sfk \,-\, 2 \,.
\]
The decompositions \eqref{E:chi-infty} and \eqref{E:chi-p} yield
\[
  \chi_\infty \ = \ \prod_{p=0}^\sfn
  (\chi^{p,0})^{\sfk-1} \, (\chi^{p,1})^{\sfk-2}
  \cdots (\chi^{p,\sfk-3})^2\,\chi^{p,\sfk-2} \,.
\]
The goal is to show that 
\begin{equation}\label{E:evgoal}
  |\chi_\infty| \ = \ \prod_{p=0}^\sfn 
  |\chi^{p,0}|^{\sfk-1} \, |\chi^{p,1}|^{\sfk-2}
  \cdots |\chi^{p,\sfk-2}| \ = \ 1 \,.
\end{equation}
It is convenient to group the terms $|\chi^{p,q}|^m$, $q+m = \sfk-1$, on the right-hand side of \eqref{E:evgoal} by their weight-graded degree $0 \le p+q \le \sfn+\sfk-2$.  To that end, we write 
\begin{equation}\label{E:evwt}
  |\chi_\infty| = \psi_0\,\psi_1 \cdots \psi_{\sfk-1} \,\cdot\, 
  \psi_{\sfk} \cdots \psi_{\sfn+\sfk-3} \,\psi_{\sfn+\sfk-2} \,,
\end{equation}
with 
\begin{eqnarray}
  \nonumber
  \psi_0 & = & |\chi^{0,0}|^{\sfk-1} \\
  \nonumber
  \psi_1 & = & 
  	|\chi^{1,0}|^{\sfk-1}\,|
	\chi^{0,1}|^{\sfk-2} \\
  \nonumber
  \psi_2 & = & |\chi^{2,0}|^{\sfk-1}\,
  	|\chi^{1,1}|^{\sfk-2}\,
	|\chi^{0,2}|^{\sfk-3}
	\\
  \label{E:ev1}
	& \vdots & \\
  \nonumber
  \psi_a & = & 
	|\chi^{a,0}|^{\sfk-1}\,
  	|\chi^{a-1,1}|^{\sfk-2}
	\cdots |\chi^{0,a}|^{\sfk-a-1} \\
  \nonumber
  & \vdots & \\
  \nonumber
  \psi_{\sfk-2} & = & 
	|\chi^{\sfk-2,0}|^{\sfk-1}\,
  	|\chi^{\sfk-3,1}|^{\sfk-2}
	\cdots |\chi^{0,\sfk-2}|^1\\
  \nonumber
  \psi_{\sfk-1} & = & 
	|\chi^{\sfk-1,0}|^{\sfk-1}\,|\chi^{\sfk-2,1}|^{\sfk-2}
	\cdots |\chi^{1,\sfk-2}|^1\,|\chi^{0,\sfk-1}|^0 \,;
\end{eqnarray}
and
\begin{eqnarray}
  \nonumber
  \psi_\sfk & = & 
    |\chi^{\sfk,0}|^{\sfk-1} \,
    |\chi^{\sfk-1,1}|^{\sfk-2} \cdots
    |\chi^{2,\sfk-2}| \cdot
    |\chi^{1,\sfk-1}\,\chi^{0,\sfk}|^0  \\
  \nonumber
  \psi_{\sfk+1} & = & 
    |\chi^{\sfk+1,0}|^{\sfk-1} \,
    |\chi^{\sfk,1}|^{\sfk-2} \cdots
    |\chi^{3,\sfk-2}| \cdot
    |\chi^{2,\sfk-1}\,\chi^{1,\sfk}\,
    	\chi^{0,\sfk+1}|^0  \\
  \nonumber
  & \vdots & \\
  \nonumber
  \psi_{\sfn-1} & = & 
  	|\chi^{\sfn-1,0}|^{\sfk-1} \, 
	|\chi^{\sfn-2,1}|^{\sfk-2} \cdots 
	|\chi^{\sfn-\sfk+1,\sfk-2}| \cdot 
	|\chi^{\sfn-\sfk,\sfk-1} \cdots \chi^{0,\sfn-1}|^0 \\
  \label{E:ev2}
  \psi_\sfn & = &
  	|\chi^{\sfn,0}|^{\sfk-1} \, 
	|\chi^{\sfn-1,1}|^{\sfk-2} \cdots 
	|\chi^{\sfk,\sfk-2}| \cdot 
	|\chi^{\sfk-1,\sfk-1} \cdots \chi^{0,\sfn}|^0 \\
  \nonumber
  \psi_{\sfn+1} & = & 
  	|\chi^{\sfn,1}|^{\sfk-2} \, 
	|\chi^{\sfn-1,2}|^{\sfk-3} \cdots 
	|\chi^{\sfk+1,\sfk-2}| \cdot 
	|\chi^{\sfk,\sfk-1} \cdots \chi^{1,\sfn}|^0 \\
  \nonumber
  & \vdots & \\
  \nonumber
  \psi_{\sfn+\sfk-3} & = & 
  |\chi^{\sfn,\sfk-3}|^2 \,|\chi^{\sfn-1,\sfk-2}| \\
  \nonumber
  \psi_{\sfn+\sfk-2} & = & |\chi^{\sfn,\sfk-2}| \,.
\end{eqnarray}
For the products in the group \eqref{E:ev1}, we see that  \eqref{E:sum} yields 
\begin{equation}\label{E:evd1}
  \psi_\ell \ = \ |\chi_\ell|^{(2\sfk-a-2)/2} = 1 \,,
  \quad \forall \ 0 \le \ell \le \sfk-1 \,.
\end{equation}
Turning to the second group \eqref{E:ev2}, the first equation of \eqref{SE:pqn-ids} implies 
\begin{equation}\label{E:evd2}
  \psi_\sfn \ = \ 1 \,.
\end{equation}
And the second equation of \eqref{SE:pqn-ids} implies
\begin{eqnarray*}
  \psi_\sfk \, \psi_{\sfn+\sfk-2} & = & 
    |\chi^{\sfk,0} \, \chi^{\sfk-1,1}|^{\sfk-2} \cdot
    |\chi^{\sfk-2,2}|^{\sfk-3} \cdots
    |\chi^{2,\sfk-2}| \cdot
    |\chi^{1,\sfk-1}\,\chi^{0,\sfk}|^0  \\
  \psi_{\sfk+1} \, \psi_{\sfn-\sfk-3} & = & 
    |\chi^{\sfk+1,0}\,\chi^{\sfk,1}\,\chi^{\sfk-1,2}|^{\sfk-3} 
    \cdot |\chi^{\sfk-2,3}|^{\sfk-4} \cdots |\chi^{3,\sfk-2}| 
    \cdot |\chi^{2,\sfk-1}\,\chi^{1,\sfk}\,
    	\chi^{0,\sfk+1}|^0  \\
  & \vdots & \\
  \psi_{\sfn-1}\,\psi_{\sfn+1} & = & 
  	|\chi^{\sfn-1,0} \, \chi^{\sfn-2,1} \cdots 
	\chi^{\sfn-\sfk+1,\sfk-2}|  \,.
\end{eqnarray*}
It follows from \eqref{E:sum} that 
\begin{equation} \label{E:evd3}
\renewcommand{\arraystretch}{1.3}
\begin{array}{rcl}
  \psi_\sfk \, \psi_{\sfn+\sfk-2} & = & 
  |\chi_\sfk|^{(\sfk-2)/2} \ = \ 1 \\
  \psi_{\sfk+1}\,\psi_{\sfn+\sfk-3} & = & 
  |\chi_{\sfn+\sfk-3}|^{(\sfk-3)/2} \ = \ 1  \\
  & \vdots & \\
  \psi_{\sfn-1}\,\psi_{\sfn+1} & = & 
  |\chi_{\sfn-1}|^{1/2} \ = \ 1 \,.
\end{array}
\end{equation}
The desired \eqref{E:evgoal} now follows from \eqref{E:evwt}, \eqref{E:evd1}, \eqref{E:evd2} and \eqref{E:evd3}.  \hfill\qed

%----------------------------------------------------------
\subsection{Proof of Theorem \ref{T:normone}: odd weight} \label{S:prf-od}
%----------------------------------------------------------

If $\sfn \ge 3$ is odd, then $\sfn = 2 \sfk - 1$.  The decompositions \eqref{E:chi-infty} and \eqref{E:chi-p} yield
\[
  \chi_\infty \ = \ \prod_{p=0}^\sfn
  (\chi^{p,0})^\sfn \, (\chi^{p,1})^{\sfn-2}
  \cdots (\chi^{p,\sfk-2})^3\,\chi^{p,\sfk-1} \,.
\]
The goal is to show that
\begin{equation}\label{E:odgoal}
  |\chi_\infty| \ = \ \prod_{p=0}^\sfn
  |\chi^{p,0}|^\sfn \, |\chi^{p,1}|^{\sfn-2}
  \cdots |\chi^{p,\sfk-2}|^3\,|\chi^{p,\sfk-1}| \ = \ 1 \,.
\end{equation}
Again, it is convenient to group the terms $|\chi^{p,q}|^m$, $2q+m = \sfn$, on the right-hand side of \eqref{E:odgoal} by their weight-graded degree $0 \le p+q \le \sfn+\sfk-1$.  We write 
\begin{equation}\label{E:odwt}
  |\chi_\infty| \ = \ 
  \psi_0\,\psi_1 \cdots \psi_{\sfk-1} \,\cdot\,
  \psi_{\sfk} \cdots \psi_{\sfn+\sfk-2} \,\psi_{\sfn+\sfk-1}\,,
\end{equation}
with 
\begin{eqnarray}
  \nonumber
  \psi_0 & = & |\chi^{0,0}|^\sfn \\
  \nonumber
  \psi_1 & = & 
  	|\chi^{1,0}|^\sfn\,|\chi^{0,1}|^{\sfn-2} \\
  \nonumber
  \psi_2 & = & |\chi^{2,0}|^\sfn\,
  	|\chi^{1,1}|^{\sfn-2}\,
	|\chi^{0,2}|^{\sfn-4}
	\\
  \label{E:od1}
	& \vdots & \\
  \nonumber
  \psi_a & = & 
	|\chi^{a,0}|^\sfn\,
  	|\chi^{a-1,1}|^{\sfn-2}
	\cdots |\chi^{0,a}|^{\sfn-2a} \\
  \nonumber
  & \vdots & \\
  \nonumber
  \psi_{\sfk-1} & = & 
	|\chi^{\sfk-1,0}|^\sfn\,
  	|\chi^{\sfk-2,1}|^{\sfn-2}
	\cdots |\chi^{0,\sfk-1}| \,;
\end{eqnarray}
and
\begin{eqnarray}
  \nonumber
  \psi_\sfk & = & 
    |\chi^{\sfk,0}|^\sfn \,
    |\chi^{\sfk-1,1}|^{\sfn-2} \cdots
    |\chi^{1,\sfk-1}| \cdot
    |\chi^{0,\sfk}|^0  \\
  \nonumber
  \psi_{\sfk+1} & = & 
    |\chi^{\sfk+1,0}|^\sfn \,
    |\chi^{\sfk,1}|^{\sfn-2} \cdots
    |\chi^{2,\sfk-1}| \cdot
    |\chi^{1,\sfk}\,\chi^{0,\sfk+1}|^0  \\
  \nonumber
  & \vdots & \\
  \nonumber
  \psi_{\sfn-1} & = & 
  	|\chi^{\sfn-1,0}|^\sfn \, 
	|\chi^{\sfn-2,1}|^{\sfn-2} \cdots 
	|\chi^{\sfk-1,\sfk-1}| \cdot 
	|\chi^{\sfk-2,\sfk} \cdots \chi^{0,\sfn-1}|^0 \\
  \label{E:od2}
  \psi_\sfn & = &
  	|\chi^{\sfn,0}|^\sfn \, 
	|\chi^{\sfn-1,1}|^{\sfn-2} \cdots 
	|\chi^{\sfk,\sfk-1}| \cdot 
	|\chi^{\sfk-1,\sfk} \cdots \chi^{0,\sfn}|^0 \\
  \nonumber
  \psi_{\sfn+1} & = & 
  	|\chi^{\sfn,1}|^{\sfn-2} \, 
	|\chi^{\sfn-1,2}|^{\sfn-4} \cdots 
	|\chi^{\sfk+1,\sfk-1}| \cdot 
	|\chi^{\sfk,\sfk} \cdots \chi^{1,\sfn}|^0 \\
  \nonumber
  & \vdots & \\
  \nonumber
  \psi_{\sfn+\sfk-2} & = & |\chi^{\sfn,\sfk-2}|^3 \,
  	|\chi^{\sfn-1,\sfk-1}| \\
  \nonumber
  \psi_{\sfn+\sfk-1} & = & |\chi^{\sfn,\sfk-1}| \,.
\end{eqnarray}
For the terms in the first group \eqref{E:od1}, \eqref{E:sum} yields
\begin{equation}\label{E:odd1}
  \psi_\ell \ = \ |\chi_\ell|^{\sfn-\ell} = 1 \,,
  \quad \forall \ 0 \le \ell \le \sfk-1 \,.
\end{equation}
Turning to the second group \eqref{E:od2}, the first equation of \eqref{SE:pqn-ids} implies 
\begin{equation}\label{E:odd2}
  \psi_\sfn \ = \ 1 \,.
\end{equation}
For the remaining terms in this group, the second equation of \eqref{SE:pqn-ids} yields
\begin{eqnarray*}
  \psi_\sfk \,\psi_{\sfn+\sfk-1} & = & 
  	|\chi^{\sfk,0}|^{\sfn-1} \, \Big\{
	|\chi^{\sfk-1,1}|^{\sfn-2} \,
	|\chi^{\sfk-2,2}|^{\sfn-4}
	\cdots \\
	& & \hspace{110pt} \cdots |\chi^{2,\sfk-2}|^3 \,
	|\chi^{1,\sfk-1}|^1 \Big\} \,
    |\chi^{0,\sfk}|^0 \\
  \psi_{\sfk+1}\,\psi_{\sfn+\sfk-2} & = &
  |\chi^{\sfk+1,0} \, \chi^{\sfk,1}|^{\sfn-3} 
  \, \Big\{
    |\chi^{\sfk-1,2}|^{\sfn-4} \,,
    |\chi^{\sfk-2,3}|^{\sfn-6} \cdots \\
  & & \hspace{110pt} \cdots
    |\chi^{3,\sfk-2}|^3 
    |\chi^{2,\sfk-1}|^1
  \, \Big\}  \,
    |\chi^{1,\sfk}\,\chi^{0,\sfk+1}|^0 \\
  & \vdots & \\
  \psi_{\sfn-1} \,\psi_{\sfn+1} & = &   
  |\chi^{\sfn-1,0} \,\chi^{\sfn-2,1} \cdots 
	\chi^{\sfk,\sfk-2}|^2 \, 
	|\chi^{\sfk-1,\sfk-1}| \cdot 
	|\chi^{\sfk-2,\sfk} \cdots \chi^{0,\sfn-1}|^0
\end{eqnarray*}
It follows from \eqref{E:sum} that 
\begin{equation} \label{E:odd3}
\renewcommand{\arraystretch}{1.3}
\begin{array}{rcl}
  \psi_\sfk\,\psi_{\sfn+\sfk-1} & = & 
  |\chi_{\sfn+\sfk-1}|^{\sfk-1} \ = \ 1 \\
  \psi_{\sfk+1}\,\psi_{\sfn+\sfk-2} & = & 
  |\chi_{\sfn+\sfk-2}|^{\sfk-2} \ = \ 1  \\
  & \vdots & \\
  \psi_{\sfn-1}\,\psi_{\sfn+1} & = & 
  |\chi_{\sfn+1}| \ = \ 1 \,.
\end{array}
\end{equation}
The desired \eqref{E:odgoal} now follows from \eqref{E:odwt}, \eqref{E:odd1}, \eqref{E:odd2} and \eqref{E:odd3}. \hfill\qed

%----------------------------------------------------------
\section{Proof of Theorem \ref{T:chi-infty}} \label{S:prf-rtu}
%----------------------------------------------------------

%----------------------------------------------------------
\subsection{Basic idea of the proof} \label{S:idea}
%----------------------------------------------------------

We will construct a rational representation $\monzero \to \tGL(P)$, and a $\gamma$--invariant rational subspace $P' \subset P$ with the property that the $\gamma$--eigenvalues of $P'$ include $\overline{\chi_\infty(\gamma)}{}^{-1}$, and all have absolute value one (Lemma \ref{L:U}).  Since these eigenvalues are the roots of a rational polynomial (the characteristic polynomial of $\gamma : P' \to P'$), it will then follow from Kronecker's Theorem that these eigenvalues are roots of unity. 

In the special case that there exists a $\monzero$--invariant $N$ polarizing $(W,F)$, standard constructions yield the rational representation $\monzero \to \tGL(P)$ and $P' = P$; this is reviewed in \S\ref{S:spec-case}.  In the general case, a new perspective is needed to deal with this lack of invariance, cf.~\S\S\ref{S:gen-case1}--\ref{S:gen-case2}.

%----------------------------------------------------------
\subsection{Characters defined by the mixed Hodge structure on $H$} \label{S:H}
%----------------------------------------------------------

Recall the induced mixed Hodge structure on $H$ of \S\ref{S:indH}.  Just as we used the mixed Hodge structure on $V$ to define the characters $\chi^{p,q} : \Gamma_\infty \to \bC^*$, we may use the induced mixed Hodge structure on $H$ to define characters $\eta^{p,q} : \Gamma_\infty \to \bC^*$.  To be precise: By \eqref{E:monW} and \eqref{SE:ind-MHS}, the induced action of $\monzero$ on $H$ preserves the subspaces $W_\ell(H)$.  In particular, there is an induced action of $\monzero$ on $\tGr^W_\ell(H) = W_\ell(H)/W_{\ell-1}(H)$.  Note that $\tGr^W_\ell(H)$ is a rational $\monzero$--representation.  Since $(W,F)$ is a mixed Hodge structure, $F$ induces a Hodge decomposition 
\begin{equation} \nonumber %\label{E:GrW(H)pq}
  \tGr^W_\ell(H_\bC) \ = \
  \bigoplus_{p+q=\ell} \tGr^W_\ell(H)^{p,q} \,.
\end{equation}
Recall that the action of $\monzero$ preserves the Hodge summands $\tGr^W_\ell(V)^{p,q}$ (Lemma \ref{L:Gr-pq-preserved}).  The induced mixed Hodge structure on $H$, cf.~\S\ref{S:indH}, inherits this property: the action of $\monzero$ on $\tGr^W_\ell(H_\bC)$ preserves the summands $\tGr^W_\ell(H)^{p,q}$.  So we may define a character $\eta^{p,q} : \monzero \to \bC^*$ by
\[
  \eta^{p,q}(\gamma) \ = \ \tdet\{ 
  	\gamma : \tGr^W_\ell(H)^{p,q} \ \to \ \tGr^W_\ell(H)^{p,q} 
	\} \,.
\]
We have natural identifications $\tGr^W_\ell(H)^{p,q} \simeq H^{p,q}_{W,F}$, and it follows from \eqref{E:dsLinfty} that 
\begin{equation}\label{E:chi-infty-eta}
  \chi_\infty \ = \ \eta^{\sfw-\sfb,\sfw-\sfa} \,.
\end{equation}
The analogs of \eqref{SE:pq-ids} hold: Since $\overline{\tGr^W_\ell(H)^{p,q}} = \tGr^W_\ell(H)^{q,p}$, we have
\begin{equation} \label{E:eta-conj}
  \overline{\eta^{p,q}} \ = \ \eta^{q,p} \,.
\end{equation}
The nondegeneracy of $Q$ and the $Q$--isotropy of the weight filtration $W(H)$, \cf~\eqref{E:indQW}, identify $\tGr^W_\ell(H)^\vee \simeq  \tGr^W_{2\sfw-\ell}(H)$.  Then the $Q$--isotropy of the Hodge filtration $F(H_\bC)$, \cf~\eqref{E:indQF}, identifies 
\begin{equation}\label{E:Gr-dual}
  (\tGr^W_\ell(H)^{p,q})^\vee \ \simeq \ 
  \tGr^W_{2\sfw-\ell}(H)^{\sfw-p,\sfw-q}
\end{equation}
so that 
\begin{equation} \label{E:eta-dual}
  (\eta^{p,q})^{-1} \ = \ \eta^{\sfw-p,\sfw-q} \,.
\end{equation}
From \eqref{E:chi-infty-eta}, \eqref{E:eta-conj} and \eqref{E:eta-dual} we see that 
\begin{equation} \label{E:eta-ab}
  (\overline{\chi_\infty})^{-1} \ = \ \eta^{\sfa,\sfb} \,.
\end{equation}

%----------------------------------------------------------
\subsection{Proof of Theorem \ref{T:chi-infty}: special case that there exists a $\monzero$--invariant $N$} \label{S:spec-case}
%----------------------------------------------------------

Before proving the theorem in the general case, it is instructive to consider the following special case.  Suppose that $(W,F)$ can be polarized by an element $N \in \fg_\bQ \cap \fg^{-1,-1}_{W,F}$ that is \emph{invariant} under $\monzero$.  The first containment of \eqref{E:HN(WF)} implies 
\begin{equation} \label{E:NGrH}
  N : \tGr^W_\ell(H) \to \tGr^W_{\ell-2}(H) \,.
\end{equation}
The kernel $P_{\sfw+\ell}(N) = \tker\{N^{\ell+1} : \tGr^W_{\sfw+\ell}(H) \to \tGr^W_{\sfw-\ell-2}(H)\}$ is $\monzero$--invariant.  

\begin{lemma}
Given $\gamma \in \monzero$, the eigenvalues of $\gamma : P_{\sfw+\ell}(N) \to P_{\sfw+\ell}(N)$, all have absolute value one. 
\end{lemma}

\begin{proof}
The subspace $P_{\sfw+\ell}(N) \subset \tGr^W_{\sfw+\ell}(H)$ is a Hodge substructure of $\tGr^W_{\sfw+\ell}(H)$, and the action of $\monzero$ preserves the Hodge summands $P_{\sfw+\ell}(N)^{p,q} = P_{\sfw+\ell}(N)_\bC \,\cap\, \tGr^W_{\sfw+\ell}(H)^{p,q}$.  (Cf.~\S\ref{S:not-LMHS} and \S\ref{S:indN}.)  So it suffices to show that the eigenvalues of $\gamma : P_{\sfw+\ell}(N)^{p,q} \to P_{\sfw+\ell}(N)^{p,q}$ all have absolute value one, for all $p,q$.  

Suppose that $v \in \tGr^W_{\sfw+\ell}(H)^{p,q}$ is an eigenvector with eigenvalue $\e$.  Then $\overline v \in \tGr^W_{\sfw+\ell}(H)^{q,p}$ is an eigenvector with eigenvalue $\overline\e$, and $N^\ell(\overline v) \in \tGr^W_{\sfw-\ell}(H)^{q-\ell,p-\ell}$ is an eigenvector with eigenvalue $\overline\e$.  The Hodge structure on $P_{\sfw+\ell}(N)$ is polarized by $Q(\cdot , N^\ell \cdot)$, so that $Q(v,N^\ell\overline v) \not=0$.  Since $Q$ is $\Gamma$--invariant, and $N$ is $\monzero$--invariant we have 
\[
  0 \ \not= \ Q(v,N^\ell\overline v) \ = \ 
  Q(\gamma v, \gamma N^\ell\overline v) \ = \ 
  |\e|^2\, Q(v,N^\ell\overline v) \,,
\]
for all $\gamma \in \monzero$.  This forces $|\e|^2=1$.
\end{proof}

\begin{corollary}\label{C:None}
Given $\gamma \in \monzero$, the eigenvalues of $\gamma : P_{\sfw+\ell}(N) \to P_{\sfw+\ell}(N)$, are roots of unity. 
\end{corollary}

\begin{proof}
Since $N$ is rational, the representation $\monzero \to \tGL(P_{\sfw+\ell}(N))$ is rational.  Kronecker's Theorem implies these eigenvalues are all roots of unity.
\end{proof}

\begin{corollary}
The eigenvalue $\chi_\infty(\gamma)$ is a root of unity.
\end{corollary}

\begin{proof}
The natural identification $\tGr^W_\ell(H)^{p,q} \simeq H^{p,q}_{W,F}$ and \eqref{SE:dsL} allow us to identify 
\begin{equation}\label{E:Lids}
  L \ = \ H^{\sfa,\sfb}_{W,F} 
  \ \simeq \ \tGr^W_{\sfw+\sfc}(H)^{\sfa,\sfb} \,.
\end{equation}
Elements $\gamma \subset \monzero$ act on $\tGr^W_{\sfw+\sfc}(H)^{\sfa,\sfb}$ by the eigenvalue $\eta^{\sfa,\sfb}(\gamma)$.  So if we can show $\tGr^W_{\sfw+\sfc}(H)^{\sfa,\sfb} \subset P_{\sfw+\sfc}(N)$, then the corollary will follow from \eqref{E:eta-ab} and Corollary \ref{C:None}.  To see that $\tGr^W_{\sfw+\sfc}(H)^{\sfa,\sfb} \subset P_{\sfw+\sfc}(N)$, note that the definition of $\sfa,\sfb$ implies $0 = H^{\sfw-\sfb-1,\sfw-\sfa-1}_{W,F} \supset N^{\sfc+1}(H^{\sfa,\sfb}_{W,F})$, \cf~Remark \ref{R:ab}.  
\end{proof}

%----------------------------------------------------------
\subsection{Proof of Theorem \ref{T:chi-infty}: the general case, step 1} \label{S:gen-case1}
%----------------------------------------------------------

The main difficulty that must be handled is the fact that there may be no $\monzero$--invariant $N$ that polarizes the mixed Hodge structure $(W,F)$.  Let $\cM(A)$ be the set of all limiting mixed Hodge structures $(W,F,\s)$ arising along $A$, as in \S\ref{S:LMHS}, and set 
\[
  \cN \ = \ \{ \tAd_\gamma(N) \ | \ \gamma \in \monzero \,,\ 
  N \in \s \,,\ (W,F,\s) \in \cM(A) \}
  \ \subset \ \fg_\bQ \,. 
\]

\begin{lemma} \label{L:cN}
We have $\displaystyle \cN \ \subset \ 
  \bigoplus_{p,q\le-1} \fg^{p,q}_{W,F}$.
\end{lemma}

\begin{proof}
The weight filtration $W$ is well-defined along $A$, by definition of $\PhiSBB$.  Proposition \ref{P:RvE} and \eqref{E:Finfty} imply 
any two Hodge filtrations $F$ and $F'$ arising along $A$ lie in the same $\tStab_{G_\bC}(F_\infty)$ orbit.  The identities \eqref{E:ds-conj} and \eqref{E:FvFinfty} imply that the images of $\tStab_{G_\bC}(F_\infty)$ and $S$ coincide in $G_\bC/\tStab_{G_\bC}(F)$.  This implies that $F$ and $F'$ lie in the same $S$--orbit.  

The definition \eqref{SE:monzero} implies $S$ stabilizes
\[
  F^{-1}_\infty(\fg_\bC) \,\cap\, \overline{F^{-1}_\infty(\fg_\bC)}
  \ = \ \bigoplus_{p,q \le -1} \fg^{p,q}_{W,F} \,.
\]
The lemma now follows from \eqref{E:Nin}.
\end{proof}

In general, $\cN \not\subset \fg^{-1,-1}_{W,F}$, and so elements of $\cN$ will not ``polarize'' the mixed Hodge structure $(W,F)$ in the sense of \cite[(2.6)]{MR840721}.  Nonetheless, given $M \in \cN$, Lemma \ref{L:cN} implies we still have a well-defined isomorphism $M^\sfc : \tGr^W_{\sfw+\sfc}(H) \to \tGr^W_{\sfw-\sfc}(H)$ with the properties that 
\begin{equation}\label{E:Mc}
  M^\sfc(\tGr^W_{\sfw+\sfc}(H)^{p+\sfc,q+\sfc}) \ = \ 
  \tGr^W_{\sfw-\sfc}(H)^{p,q} \,,
\end{equation}
and $Q_M(u,v)=Q( u , M^\sfc v)$ polarizes the induced Hodge structure on $\tker\{ M^{\sfc+1} : \tGr^W_{\sfw+\sfc}(H) \to \tGr^W_{\sfw-\sfc-2}(H) \}$.  (The map $M^\sfc : \tGr^W_{\sfw+\sfc}(H) \to \tGr^W_{\sfw-\sfc}(H)$ is pictured in Figure \ref{fig:hd2}.)  And while the kernel will not be $\monzero$--invariant in general, the intersection
\[
  P \ = \ 
  \bigcap_{M \in \cN} 
  \tker\{ M^{\sfc+1} : 
  	\tGr^W_{\sfw+\sfc}(H) \to \tGr^W_{\sfw-\sfc-2}(H) \} 
  \ \subset \ \tGr^W_{\sfw+\sfc}(H)
\]
is a rational $\monzero$--submodule.   Lemma \ref{L:cN} implies that $P$ is a Hodge substructure of $\tGr^W_{\sfw+\sfc}(H)$; and every $Q_M$, $M \in \cN$, polarizes this Hodge structure.

\begin{lemma}\label{L:inP}
We have $L \subset P_\bC$.
\end{lemma}

\begin{remark}\label{R:Lids}
In the statement of the lemma we are using \eqref{E:Lids} to identify $L$ with the subspace $\tGr^W_{\sfw+\sfc}(H)^{\sfa,\sfb} \subset \tGr^W_{\sfw+\sfc}(H)$.
\end{remark}

\begin{proof}
Lemma \ref{L:cN} and the definition of $\sfa,\sfb$ imply $0 = H^{\sfw-\sfb-1,\sfw-\sfa-1}_{W,F} \supset N^{\sfc+1}(H^{\sfa,\sfb}_{W,F})$, \cf~Remark \ref{R:ab}.  The lemma follows from the identification \eqref{E:Lids}.
\end{proof}

\begin{figure}
\begin{tikzpicture}[baseline=(current  bounding  box.center),scale=0.8]
  \draw [gray,dashed] (2.5,6.5) --  (6.5,2.5);
  \node [gray,below right] at (6.5,2.5) {$\tGr^W_{\sfw+\sfc}(H)$};
  \draw [gray,dashed] (-0.5,3.5) --  (3.5,-0.5);
  \node [gray,above left] at (-0.5,3.5) {$\tGr^W_{\sfw-\sfc}(H)$};
  \node [gray] at (5.7,2.3) {$M$};
  \draw [gray] (0,0) -- (6,0) -- (6,6) -- (0,6) -- (0,0);
  \draw [gray] (1,1) -- (5,1) -- (5,5) -- (1,5) -- (1,1);
  \foreach \x in {1,...,5}
  {
  	\foreach \y in {1,...,5}
	{
	  	\draw [fill,gray] (\x,\y) circle [radius=0.07];
	}
  }
  \draw [gray,-to] (5.8,2.8) -- (5.2,2.2);
  \draw [gray,-to] (4.8,1.8) -- (4.2,1.2);
  \draw [gray,-to] (3.8,0.8) -- (3.2,0.2);
  \draw [gray,-to] (4.8,3.8) -- (4.2,3.2);
  \draw [gray,-to] (3.8,2.8) -- (3.2,2.2);
  \draw [gray,-to] (2.8,1.8) -- (2.2,1.2);
  \draw [gray,-to] (3.8,4.8) -- (3.2,4.2);
  \draw [gray,-to] (2.8,3.8) -- (2.2,3.2);
  \draw [gray,-to] (1.8,2.8) -- (1.2,2.2);
  \draw [gray,-to] (2.8,5.8) -- (2.2,5.2);
  \draw [gray,-to] (1.8,4.8) -- (1.2,4.2);
  \draw [gray,-to] (0.8,3.8) -- (0.2,3.2);
  \draw [fill] (3,0) circle [radius=0.08];
  %\node [below left] at (3,0) {$(\sfa,\sfb)$};
  \node [below left] at (3,0) {$L_\infty$};
  \draw [fill] (0,3) circle [radius=0.08];
  %\node [below left] at (0,3) {$(\sfb,\sfa)$};
  \node [below left] at (0,3) {$\overline{L_\infty}$};
  \draw [fill] (6,3) circle [radius=0.08];
  %\node [below right] at (6,3) {$(\sfa+k,\sfb+k)$};
  \node [above right] at (6,3) {$L$};
  \draw [fill] (3,6) circle [radius=0.08];
  %\node [below right] at (3,6) {$(\sfb+k,\sfa+k)$};
  \node [above right] at (3,6) {$\overline{L}$};
\end{tikzpicture}
\label{fig:hd2}
\caption{The map $M^\sfc : \tGr^W_{\sfw+\sfc}(H) \to \tGr^W_{\sfw-\sfc}(H)$.}
\end{figure}

%----------------------------------------------------------
\subsection{Proof of Theorem \ref{T:chi-infty}: the general case, step 2} \label{S:gen-case2}
%----------------------------------------------------------

A priori $Q_M$ need not be $\monzero$--invariant: we have $(\gamma \cdot Q_M)(u,v) = Q_M(\gamma^{-1} u , \gamma^{-1}v) = Q(\gamma^{-1} u , M^\sfc \gamma^{-1} v)$.  Since $Q$ is $\monzero$--invariant, this yields $\gamma \cdot Q_M = Q_{\tAd_\gamma(M)}$.   In the absence of a monodromy-invariant polarization on $P$ we proceed as follows.  The bilinear form $Q_M$ is an element of the rational $\monzero$--submodule $T = \tGr^W_{\sfw+\sfc}(H)^\vee \ot \tGr^W_{\sfw+\sfc}(H)^\vee$.  Fix $\gamma \in \monzero$, and let $\tilde T_\nu \subset T\ot\bC$ be the generalized eigenspaces, as $\nu$ runs over the \emph{distinct} eigenvalues of $\gamma : T \to T$.  

\begin{lemma}
Both $\tilde T_1$ and $\op_{\n\not=1}\,\tilde T_\nu$ are defined over $\bQ$.
\end{lemma}

\begin{proof}
Let $p_\gamma(t) = \tdet\{ t\,\tId - \gamma : T \to T \} \in \bQ[t]$ be the characteristic polynomial.  Let $0 \le c \in \bZ$ be the largest integer with the property that $(t-1)^c$ divides $p_\gamma(t)$.  Set $q(t) = p_\gamma(t)/(t-1)^k \in \bQ[t]$.  Then $p_\gamma(t) = (t-1)^c\,q(t)$.  And both
\begin{eqnarray*}
  \tilde T_1 & = & 
  \tker\{ (\gamma - \tId)^c : T \ot \bC \ \to \ T \ot \bC \} \\
  \op_{\nu\not=1}\, \tilde T_\nu & = & 
  \tker\{ q(\gamma) : T \ot \bC \ \to \ T \ot \bC \} \,.
\end{eqnarray*}
are defined over $\bQ$.
\end{proof}

Let $T = T_1 \op T'$ be the associated rational decomposition: $T_1 = T \cap \tilde T_1$ and $T' = T \cap \op_{\nu\not=1} \tilde T_\nu$.  Let $Q_M = Q_{M,1} + Q_M'$ be the corresponding sum.  By construction, we have $\gamma \cdot Q_{M,1} = Q_{\tAd_\gamma M,1}$ and $\gamma \cdot Q_M' = Q_{\tAd_\gamma M}'$.  So
\[
  P' \ = \ \bigcap_{M \in \cN} 
  \{ u \in P \ | \  Q_M'(u,P) \ = \ 0 \} 
\]
is a rational subspace of $P$ that is preserved under the action of $\gamma$.

\begin{lemma}\label{L:U}
The eigenvalues of $\gamma : P' \to P'$ include $\overline{\chi_\infty(\gamma)}{}^{-1}$ and all have absolute value one.
\end{lemma}

\noindent As discussed in \S\ref{S:idea}, this establishes Theorem \ref{T:chi-infty}, and by extension Theorem \ref{T:triv} (Remark \ref{R:chibeta}).

\begin{proof}
It follows from \eqref{E:ds-indQ} that the polarization $Q$ identifies $\tGr^W_{\sfw+\sfc}(H)^\vee \simeq \tGr^W_{\sfw-\sfc}(H)$, so that $T \simeq \tGr^W_{\sfw-\sfc}(H) \ot \tGr^W_{\sfw-\sfc}(H)$.  Let $E_{\sfw-\sfc}$ denote the set of $\gamma$--eigenvalues of $\tGr^W_{\sfw-\sfc}(H)$ listed with (algebraic) multiplicity.  We may fix a basis $\{e_\m\}_{\m\in E_{\sfw-\sfc}}$ of $\tGr^W_{\sfw-\sfc}(H)$ so that each $e_\m$ is a (generalized) $\gamma$--eigenvector with eigenvalue $\m$.  Then 
\[
  Q_M \ = \ \sum_{\m,\n} q_{M\m\n}\, e_\m\ot e_\n
\]
defines $q_{M\m\n}\in \bC$; and we have 
\[
  Q_{M,1} \ = \ \sum_{\m\n=1} q_{M\m\n}\, e_\m\ot e_\n
  \tand 
  Q_M' \ = \ \sum_{\m\n\not=1} q_{M\m\n}\, e_\m\ot e_\n \,.
\]

Since the action of $\gamma$ on $\tGr^W_{\sfw-\sfc}(H_\bC)$ preserves the Hodge summands $\tGr^W_{\sfw-\sfc}(H)^{p,q}$, cf.~\S\ref{S:H}, we may choose this basis so that each $e_\m$ is contained in some $\tGr^W_{\sfw-\sfc}(H)^{p,q}$.  Since $\gamma$ is rational, and therefore real, we may also choose this basis so that it is closed under complex conjugation: for every $\m$, there exists $\olm$ so that $\overline{e_\m} = e_{\olm}$.   Let $\{f_\m\}_{\n\in E_{\sfw-\sfc}}$ be the dual basis of $\tGr^W_{\sfw+\sfc}(H)$ defined by $Q(e_\m,f_\n) = \d_{\m\n}$.  Then $f_\m$ is a (generalized) eigenvector with eigenvalue $\m^{-1}$.  And if $e_\m \in \tGr^W_{\sfw-\sfc}(H)^{p,q}$, then \eqref{E:ds-indQ} implies $f_\m \in \tGr^W_{\sfw+\sfc}(H)^{\sfw-p,\sfw-q}$.  We have 
\[
  q_{M\m\n} \ = \ Q_M(f_\m,f_\n) \,.
\]

We claim that $L \simeq \tGr^W_{\sfw+\sfc}(H)^{\sfa,\sfb} \subset P'$, cf.~Remark \ref{R:Lids}.  First note that \eqref{E:eta-ab}, Lemma \ref{L:inP} and \eqref{E:Lids}  imply that $L$ is a $\gamma$--eigenline of $P_\bC$ with eigenvalue $\d=\overline{\chi_\infty(\gamma)}{}^{-1}$.  Let $f_\d$ be the basis vector spanning $\tGr^W_{\sfw+\sfc}(H)^{\sfa,\sfb} \simeq L$.  Then \eqref{E:Gr-dual} implies $Q_M(f_\d,f_\n)$ is nonzero if and only if $M^\sfc(f_\n) \in \tGr^W_{\sfw-\sfa,\sfw-\sfb}(H)$.  Lemma \ref{L:cN} implies this is equivalent to $f_\n \in \tGr^W_{\sfw+\sfc}(H)^{\sfb,\sfa}$.  This in turn forces $f_\n = f_{\overline\d}$.   It then follows from Theorem \ref{T:normone} that $q_{M\d\overline\d} \,e_\d \ot e_{\overline\d}$ is a summand of $Q_{M,1 }$, and $Q_M'(f_\d,P) = 0$.  Thus $L \subset P'$, and $\d=\overline{\chi_\infty(\gamma)}{}^{-1}$ is an eigenvalue of $\gamma : P' \to P'$.

It follows from Lemma \ref{L:cN} that $P' \subset P \subset \tGr^W_{\sfw+\sfc}(H)$ is a Hodge substructure.  So we may assume without loss of generality that the basis $\{e_\m\}$ was chosen so that $\{ f_\m\}_{\m \in E(P')}$ spans $P'$, and $\{ f_\m\}_{\m \in E(P)}$ spans $P$, for some subsets $E(P') \subset E(P) \subset E_{\sfw-\sfc}$.  Now suppose that $f_\m \in P^{r,s}$.  Since $Q_M$ polarizes the Hodge structure on $P$, we have $\bi^{r-s} Q_M(f_\m,\overline{f_\m})  > 0$.  So if $f_\n \in (P')^{r,s}$, then $q_{M\n\oln} = Q_{M,1}(f_\n,f_{\oln}) \not=0$.  It follows that $|\n|^2$ is an eigenvalue of $\gamma : T_1 \to T_1$; that is, $|\n|^2=1$.
\end{proof}

%----------------------------------------------------------
\section{Discussion of Conjecture \ref{conj:GGLR}} \label{S:dconj}
%----------------------------------------------------------

\begin{theorem} \label{T:conj}
Conjecture \ref{conj:GGLR} holds if the holomorphic functions on $Z_I \cap X$ extend to holomorphic functions on $X$.
\end{theorem}

\begin{remark}
Extension results of the type desired typically require that $X$ be weakly pseudoconvex; work in this direction includes \cite{Robles-extnHnorm, Robles-pseudocnvx-eg}.
\end{remark}

%----------------------------------------------------------
\subsection{Outline of proof}
%----------------------------------------------------------

Recall that $\Le$ descends to $\olP'$ (Corollary \ref{C:descends}).  Suppose we can show that $\olP'$ is a normal Moishezon variety containing $\wp'$ as a Zariski open subset, and that the extension $\sPhiSBB$ of $\Phi'$ is a morphism of algebraic spaces.  It will then follow from \cite{GGLR} that $\Le$ is ample over $\olP'$, establishing Conjecture \ref{conj:GGLR}.

The set
\begin{equation}\label{E:eqrelR}
  R \ = \ 
  \{ (b_1,b_2) \in \olB \times \olB \ | \ \sPhiSBB(b_1) = \sPhiSBB(b_2) \}
\end{equation}
defines an equivalence relation on $\olB$ with the property that $\sPhiSBB : \olB \to \olP'$ is the quotient map.  Suppose that $R$ defines a proper, holomorphic equivalence relation on $\olB$.  Then \cite[\S3, Theorem 2]{MR719132} asserts that the quotient $\olP'$ is a compact, complex analytic variety, and the quotient map $\sPhiSBB$ is a proper holomorphic completion of $\Phi'$.  Since $\olB$ is projective (and therefore Moishezon) it follows that $\olP'$ is Moishezon \cite[\S5, Corollary 11]{MR673560}.  As Moishezon spaces are algebraic, Serre's GAGA implies $\sPhiSBB$ is a morphism, \cite[\S7]{MR260747}.

So the essential problem is to show that $R$ defines a proper, holomorphic equivalence relation.  Since $\left.\PhiSBB\right|_X$ is proper (Corollary \ref{C:top}), it suffices to show that the holomorphic functions on $X$ separate the fibres of $\left.\PhiSBB\right|_X$.

%----------------------------------------------------------
\subsection{Separation of fibres}
%----------------------------------------------------------

Bakker--Brunebarbe--Tsimerman \cite{MR4557401} have shown that $\wp = \Phi(B)$ is projectively embedded by sections of $\Xi_\mathrm{e}^{\ot m} \to \olB$ that vanish along $Z$, for a suitably large power $m$.  The arguments there apply to any line bundle of the form
\[
  \tdet(\cF^\sfn)^{\ot d_\sfn} \ot 
  \tdet(\cF^{\sfn-1})^{\ot d_{\sfn-1}} \ot \cdots \ot 
  \tdet(\cF^\sfk)^{\ot d_\sfk} \,,
\]  
with $0 < d_j \in \bZ$, including $\Le$.  We have seen that a multiple of $\Le$ is trivial over $X$ (Theorem \ref{T:triv}).  It follows that the holomorphic functions on $X$ will separate any fibre of $\left.\PhiSBB\right|_U$ from any other fibre of $\left.\PhiSBB\right|_X$.

By the same argument, $\wp_I = \PhiSBB(Z_I \cap Z_\pi)$ is projectively embedded by sections of $\Le^{\ot m} \to Z_I$ that vanish along $Z_I - (Z_I \cap Z_\pi)$, and holomorphic functions on $Z_I$ will separate any fibre of $\left.\PhiSBB\right|_{Z_I\cap Z_\pi\cap X}$ from any other fibre of $\left.\Phi^0\right|_{Z_I \cap X}$.

So to prove Conjecture \ref{conj:GGLR} it remains to show that holomorphic functions on $Z_I \cap X$ extend to all of $X$.  This completes the proof of Theorem \ref{T:conj}.  \hfill \qed

%----------------------------------------------------------
\appendix
%----------------------------------------------------------

%----------------------------------------------------------
\section{Example}
%----------------------------------------------------------

The Hilbert modular surface naturally carries Hodge theoretic interpretations \cite[III.5]{MR0457437}.  This example is a variation on the Hilbert modular surface.  The example constructed can be interpreted as parameterizing ($\Gamma$--conjugacy classes of) Hodge structures of weight one (\S\ref{S:eg-pol}), and Hodge structures of weight five (\S\ref{S:ben}).  The latter is especially interesting, as it hosts a rich collection of Hodge--theoretically meaningful line bundles, exhibiting a variety of behaviors relative to Corollary \ref{C:descends}: in some cases no (positive) power of the line bundle descends (Remark \ref{R:nopower}); in others the line bundle descends only after taking a power of four; and the line bundle $\Le$ itself, descends on the nose (Remark \ref{R:egdescends}).

%----------------------------------------------------------
\subsection{Definitions}
%----------------------------------------------------------

Set
\[
  K \ = \ \bQ(\sqrt{2},\bi) \,.
\]
The field is equipped with two pairs of conjugate complex embeddings 
\begin{equation}\label{E:embK}
  \s_1\,,\,\overline\s_1 \,,\ \s_2\,,\,\overline\s_2 
  \,:\, K \ \inj \ \bC \,.
\end{equation}

\begin{remark}\label{R:Kemb}
It will be convenient to assume that the embeddings have been indexed so that they are related as follows: if $k \in K$ and $\s_1(k) = a + b \sqrt{2} + c \bi$, with $a,b,c \in \bQ$, then $\s_2(k) = a - b \sqrt{2} + c \bi$.
\end{remark}

Define $u,\z \in K$ by 
\[
  \s_1(u) \ = \ 1 + \sqrt{2} \quad \hbox{and} \quad
  \s_1(\zeta) \ = \ \frac{\sqrt{2}}{2}(1 + \bi) \,.
\]
The ring of integers is 
\[
  \cO_K \ = \ \bZ[u,\zeta] \ = \ \bZ[\sqrt{2},\zeta] \,,
\]
cf.~\cite[\S2]{milneANT}.  The group of units $\cO_K^*$ includes both $u$ and the eighth root of unity $\z$.

%----------------------------------------------------------
\subsubsection{Multiplication on the lattice $\cO_K$} \label{S:lattmult}
%----------------------------------------------------------

Regard the ring $\cO_K$ of integers as a rank four lattice with $\bZ$--basis 
\[
  \{ \lambda_1,\lambda_2,\lambda_3,\lambda_4\} \ = \
  \{1,u , \zeta , u\zeta \} \,.  
\]
Multiplication by $\a\in\cO_K$ defines a $\bZ$--linear map 
\[
  \nu(\a)  \,:\, \cO_K \ \to \ \cO_K \,.
\] 
The matrix representation of $\nu(u) : \cO_K \to \cO_K$, with respect to the basis $\{\lambda_j\}$ is
\[
  \nu(u) \ = \ \left[ \begin{array}{cccc} 
  	0 & 1 & 0 & 0 \\ 
	1 & 2 & 0 & 0 \\
	0 & 0 & 0 & 1 \\
	0 & 0 & 1 & 2
  \end{array}\right] \,.
\]
The $\bZ$--linear maps $\nu(\a)$ define $\bC$--linear maps of the complexification $\cO_K \ot_\bZ \bC \simeq \bC^4$.  The eigenvalues of $\nu(u)$ are $u_1 = \s_1(u) = 1+\sqrt{2}$ and $u_2 = \s_2(u) = 1 - \sqrt{2}$.  And the $\nu(u)$--eigenspace decomposition of $\cO_K \ot_\bZ \bC$ is $E_{u_1} \op E_{u_2}$ with 
\begin{eqnarray*}
  E_{u_1} & = & \tspan_\bC\{ \lambda_1 + u_1\,\lambda_2 \,,\
  \lambda_3 + u_1\,\lambda_4 \} \\
  E_{u_2} & = & \tspan_\bC\{ \lambda_1 + u_2\,\lambda_2 \,,\
  \lambda_3 + u_2\,\lambda_4 \} \,.
\end{eqnarray*}
The operators $\nu(u)$ and $\nu(\zeta)$ commute.  So the $\bC$--linear action of $\nu(\zeta)$ on $\cO_K\ot_\bZ\bC$ preserves $\nu(u)$--eigenspaces $E_{u_j}$.  The matrix representation of $\nu(\zeta): E_{u_1} \to E_{u_1}$, with respect to the above basis of $\nu(u)$--eigenvectors is
\[
  \left.\nu(\z)\right|_{E_{u_1}} \ = \ 
  \left[ \begin{array}{cc}
    0 & -1 \\
    1 & (u_1-1)
  \end{array}\right] \,.
\]
So the $\nu(\z)$ eigenvalues of $E_{u_1}$ are $\z_1 = \s_1(\z) = \sqrt{2}(1+\bi)/2$ and $\overline\z_1 = \overline\s_1(\z) = \sqrt{2}(1-\bi)/2$.  The $\nu(\z)$--eigenspace decomposition is $E_{u_1} = E_{u_1,\z_1} \op E_{u_1,\overline\z_1}$ with 
\begin{eqnarray*}
  E_{u_1,\z_1} & = & \tspan_\bC\{ (\lambda_1 + u_1\,\lambda_2) 
  - \z_1(\lambda_3 + u_1\,\lambda_4) \} \\
  E_{u_1,\overline\z_1} & = & 
  \tspan_\bC\{ (\lambda_1 + u_1\,\lambda_2) 
  - \overline\z_1(\lambda_3 + u_1\,\lambda_4) \} \,.
\end{eqnarray*}
A similar computation of $\nu(\z) : E_{u_2} \to E_{u_2}$ show that the eigenvectors are $\z_2 = \s_2(\z) = - \sqrt{2}(1+\bi)/2 = -\z_1$ and $\overline\z_2 = \overline\s_2(\z) = - \sqrt{2}(1-\bi)/2 = -\overline\z_1$.  The corresponding eigenspaces are
\begin{eqnarray*}
  E_{u_2,\z_2} & = & \tspan_\bC\{ (\lambda_1 + u_2\,\lambda_2) 
  - \z_2(\lambda_3 + u_2\,\lambda_4) \} \\
  E_{u_2,\overline\z_2} & = & 
  \tspan_\bC\{ (\lambda_1 + u_2\,\lambda_2) 
  - \overline\z_2(\lambda_3 + u_2\,\lambda_4) \} \,.
\end{eqnarray*}
In summary, the simultaneous $(\nu(u),\nu(\zeta))$--eigenspace decomposition of $\cO_K \ot_\bZ \bC$ is 
\[
  \cO_K \ot_\bZ \bC \ = \ 
  E_{u_1,\z_1} \,\op\, E_{u_1,\overline\z_1} \,\op\,
  E_{u_2,\z_2} \,\op\, E_{u_2,\overline\z_2} \,.
\]

More generally, any $k \in K$ defines a $\bQ$--linear map $\nu(k) :\cO_K \ot_\bZ \bQ \to \cO_K \ot_\bZ \bQ$.  The map $\nu(k)$ commutes with both $\nu(u)$ and $\nu(\z)$.  So $\nu(k) : \cO_K \ot_\bZ \bC \to \cO_K \ot_\bZ \bC$ preserves each eigenspace $E_{u_j,\z_j}$, $E_{u_j,\overline \z_j}$, $j=1,2$.  These eigenspaces are one-dimensional; so $\nu(k)$ acts by eigenvalue.  The $\nu(k)$ eigenvalue of $E_{u_j,\z_j}$ is $\s_j(k)$, and the eigenvalue of $E_{u_j,\overline\z_j}$ is $\overline\s_j(k)$.  So it makes sense to redefine
\[
  E_{\s_j} \ = \ E_{u_j,\z_j} \tand
  E_{\overline\s_j} \ = \ E_{u_j,\overline\z_j} \,,\quad j=1,2 \,,
\]
and write
\[
  \cO_K \ot_\bZ \bC \ = \ 
  E_{\s_1} \,\op\, E_{\overline\s_1} \,\op\,
  E_{\s_2} \,\op\, E_{\overline\s_2} \,.
\]

%----------------------------------------------------------
\subsubsection{Multiplication on the lattice $V_\bZ = \cO_K\op\cO_K$}
%----------------------------------------------------------

Now regard 
\[
  V_\bZ \ = \ \cO_K \,\op\, \cO_K
\]
as a rank 8 lattice ($\simeq \bZ^8$).  As in \S\ref{S:lattmult}, multiplicition by $k \in K$ defines a $\bQ$--linear map $\mu(k) : V_\bQ \to V_\bQ$, where $V_\bQ = V_\bZ \ot_\bZ \bQ$.  If $\a \in \cO_K$, then $\mu(\a)$ preserves the lattice $V_\bZ \subset V_\bQ$.  Again, the complexification $V_\bC = V_\bZ \ot_\bZ \bC$ decomposes into a direct sum 
\begin{equation}\label{E:eigendecomp}
  V_\bC \ = \ V_{\s_1} \,\op\, V_{\overline\s_1} \,\op\,
  V_{\s_2} \,\op\, V_{\overline\s_2} 
\end{equation}
of simultaneous $\mu(k)$--eigenspaces defined by
\[
  V_{\s_j} \ = \ E_{\s_j} \,\op\, E_{\s_j} \tand
  V_{\overline\s_j} \ = \ E_{\overline\s_j} \,\op\, 
  E_{\overline\s_j} \,,\quad j=1,2\,,
\] 
as subspaces of $V_\bC \simeq (\cO_K\ot_\bZ \bC) \op (\cO_K\ot_\bZ \bC)$.

Any $g \in \tGL_2(K)$ defines a $\bQ$--linear action $g : V_\bQ \to V_\bQ$.  If $g \in \tGL_2(\cO_K)$, then the action preserves the lattice $V_\bZ$.  The action of $g$ commutes with the $\mu(k)$, and so preserves the eigenspaces $V_{\s_j}$ and $V_{\overline\s_j}$, $j=1,2$.  The induced embedding
\begin{equation}\label{E:embG}
  (\s_1,\s_2) : \tGL_2(K) \ \inj \ 
  \tGL_2(\bC) \,\times\, \tGL_2(\bC) \ \simeq \ 
  \tGL(V_{\s_1}) \,\times\, \tGL(V_{\s_2})
\end{equation}
is the matrix representation for the action of $\tGL_2(K)$ on $V_{\s_1} \op V_{\s_2}$ (with respect to the eigenbasis).

%----------------------------------------------------------
\subsubsection{Hermitian forms}
%----------------------------------------------------------

Represent elements of $K^2 = K \op K$ as column vectors $u = \left[  u_1 \ u_2 \right]^\mathsf{t}$.  Define a hermitian form of signature $(1,1)$ on $K^2$ by specifying 
\[
  h(u,v) \ = \ u{}^\mathsf{t} \, \sfh \, \overline{v} 
  \ = \ 
  u_1\,\overline{v}_2 \,+\, u_2\,\overline{v}_1\,,
\]
with matrix representation
\[
  \sfh \ = \ \left[ \begin{array}{cc} 0 & 1 \\ 1 & 0 
  \end{array} \right] \,.
\]
Define
\[
  \tU(K^2,\sfh) \ = \ \{ g \in \tGL_2(K) \ | \ 
  \overline{g}{}^\mathsf{t} \, \sfh \, g \,=\, \sfh \}
\]
and
\[
  \Gamma \ = \ \tU(K^2,\sfh) \,\cap\,\tGL_2(\cO_K) \,.
\]
Define a hermitian form $H$ of signature $(2,2)$ on $V_{\s_1}\op V_{\s_2}$ by specifying that it have matrix representation
\[
  \sfH \ = \ \left[ \begin{array}{cc|cc} 0 & 1 & &  \\ 
  	1 & 0 & & \\ \hline
	& & 0 & 1 \\
	& & 1 & 0
  \end{array} \right] \,.
\]
The restriction of the embedding \eqref{E:embG} yields
\[
  (\s_1,\s_2) \,:\, \tU(K^2,\sfh) \ \inj \ 
  \tU(1,1) \,\times\, \tU(1,1) \,.
\]  
There is a $\bQ$--structure on $G_\bR = \tU(1,1) \,\times\, \tU(1,1)$ with $\bQ$--rational points $G_\bQ = \tU(K^2,\sfh)$, and so that $\Gamma$ is an arithmetic subgroup.

%----------------------------------------------------------
\subsubsection{Polarization} \label{S:eg-pol}
%----------------------------------------------------------

Let $N \in \fgl(V_\bZ)$ be the image of 
\[
  \mathbf{n} \ = \ \left[ \begin{array}{cc}
  0 & 0 \\ 1 & 0 \end{array} \right]
  \ \in \ \fgl_2(K)
\]
under the embedding 
\[
  (\s_1,\overline\s_1 , \s_2 , \overline\s_2) :
  \fgl_2(K) \ \inj \ 
  \fgl(V_{\s_1}) \,\op\, \fgl(V_{\overline\s_1}) \,\op\,
  \fgl(V_{\s_2}) \,\op\, \fgl(V_{\overline\s_2}) \,.
\]
Define a nondegenerate, skew-symmetric bilinear form $Q : V_\bZ \,\times\, V_\bZ \ \to \ \bZ$ as follows.  Define a direct sum decomposition $V_\bC = U \op U_\infty$ by 
\[
  V_\bC \ = \ V_\bZ \ot_\bZ\bC \ \simeq \ 
  (\cO_K \ot_\bZ \bC) \ \op \ (\cO_K \ot_\bZ \bC)
  \ = \ U \ \op \ U_\infty \,.
\]
Note that the eigenspaces inherit this decomposition:
\[
  V_{\s_j} \ = \ 
  (V_{\s_j} \,\cap\,U) \ \op \ (V_{\s_j} \,\cap\,U_\infty)
  \tand
  V_{\overline\s_j} \ = \ 
  (V_{\overline\s_j} \,\cap\,U) \ \op \ 
  (V_{\overline\s_j} \,\cap\,U_\infty)
\]
We also have
\[
  \tim\,N \ = \ N(U) \ = \ U_\infty \ = \ \tker\,N \,.
\]
Define the polarization $Q$ by 
\[
  0 \ = \ \left. Q \right|_{V_{\s_1} \op V_{\s_2}}
  \ = \ \left. 
  	Q \right|_{V_{\overline\s_1} \op V_{\overline\s_2}} 
  \ = \ \left. Q \right|_U \ = \ \left. Q \right|_{U_\infty}
\]
and
\begin{eqnarray*}
  H(u,Nu) & = & Q(u,N\overline u) \ = \ -Q(N\overline u,u) \\
  & = & Q(\overline u , N u) \ = \ -Q(Nu,\overline u) \,,
  \quad \ \forall \ u \in U \,\cap\,(V_{\s_1} \op V_{\s_2}) \,.
\end{eqnarray*}

Define a second hermitian form $\sQ$ on $V_{\s_1} \op V_{\s_2}$ by specifying $\sQ(u,v) = \bi Q(u,\overline v)$.  There is a natural identification of the upper half-plane $\sH$ with $\{ \ell \in \bP(V_{\s_j}) \ | \ \sQ(\ell,\ell) > 0 \}$.  In particular, we may regard $\sH$ as parameterizing \emph{real}, weight one, $Q$--polarized Hodge structures on $V_\bR \,\cap\, (V_{\s_j} \op V_{\overline\s_j}) \simeq \bR^4$ with Hodge numbers $(2,2)$.  The product 
\[
  D \ = \ \sH \,\times\, \sH
\]
parameterizes \emph{integral}, weight one, $Q$--polarized Hodge structures on $V_\bZ \simeq \bZ^8$ with Hodge numbers $(4,4)$.  

%----------------------------------------------------------
\subsection{Discussion of Theorems \ref{T:normone} \& \ref{T:chi-infty}: weight one example} \label{eg:wt1}
%----------------------------------------------------------

Without loss of generality 
\[
  B \ = \ \Gamma \bs D
\]
is smooth.  (If necessary, replace $\Gamma$ with a finite index neat subgroup.)  The Baily--Borel compactification $B^* \supset B$ realizes $B$ as a quasi-projective variety \cite{MR0216035}.  Fix a  toroidal desingularization $\olB \to B^*$, \cite{MR0457437}.  If we take $\Phi : B \to \Gamma \bs D$ to be the identity, then $\PhiSBB$ is the map $\olB \to B^*$
\[ \begin{tikzcd}
  B \arrow[r,hook] \arrow[rd,hook]
  & \olB \arrow[d,"\PhiSBB"] \\
  & \ B^* .
\end{tikzcd} \]

The rational parabolic
\[
  P_{\infty,\bQ} \ = \ \left\{ \left.
  \left[ \begin{array}{cc} a & 0 \\ c & \overline a{}^{-1}
  \end{array} \right] \ \right| \ 
  \begin{array}{ll}
    a,c \in K \\
    \overline{a} c - a \overline{c}=0
  \end{array} \right\}
  \ \subset \ \tU(K^2,\sfh)
\]
stabilizing $U_\infty = \tim\,N$.  Let $A \subset \olB$ denote the $\PhiSBB$ fibre over the cusp.  In this case $F^1_\infty = U_\infty$, and 
\[
  \monzero \ = P_{\infty,\bQ} \,\cap\, \Gamma \ = \ 
  \left\{ \left.
  \left[ \begin{array}{cc} a & 0 \\ c & \overline a{}^{-1}
  \end{array} \right] \ \right| \ 
  \begin{array}{ll}
    a,a^{-1},c \in \cO_K \\
    \overline{a} c - a \overline{c}=0
  \end{array} \right\} \,.
\]
The pair $(N,F^1 = U)$ defines a limiting mixed Hodge structure $(W,F,N)$, as in \S\ref{S:LMHS}, with 
\[
  F_\infty \ = \ \lim_{y\to \infty} \exp(\bi yN) \cdot F \,.
\]
The limiting mixed Hodge structures arising along $A$ are all of the form $(W,F',N')$ with weight filtration $W = W(N) = W(N')$ given by $W_0(V) = W_1(V) = U_\infty$ and $W_2(V) = V$; and Hodge filtration $F'$ lying in the $P_{\infty,\bC}$--orbit of $F$.

We have $\Le = \tdet(\cFe^1)$.  The character $\chi_\infty : \Gamma_\infty \to \bC^*$ of \S\ref{S:chi-infty} is $\chi_\infty(g) = \tdet\{ g : U_\infty \to U_\infty\}$.  If 
\[
  g \ = \ \left[ \begin{array}{cc} a & 0 \\ c & \overline a{}^{-1}
  	 \end{array} \right] \ \in \ \monzero \,,
\]
then $\chi_\infty(g) = |\sigma_1(a)\,\sigma_2(a)|^{-2}$.  It follows from \cite[\S4.4, Proposition 1]{MR265266} that $\chi_\infty(g) = 1$.  This verifies both Theorems \ref{T:normone} and \ref{T:chi-infty}.

%----------------------------------------------------------
\subsection{Discussion of Theorems \ref{T:normone} \& \ref{T:chi-infty}: weight five example} \label{S:ben}
%----------------------------------------------------------

This example is a ``twist'' of the previous.  As discussed in \S\ref{S:eg-pol}, one of the factors $\sH$ in $D = \sH \times \sH$ parameterizes real, weight one, $Q$--polarized Hodge structures on $V_\bR \,\cap\, (V_{\s_2} \op V_{\overline\s_2})$ with Hodge numbers $(2,2)$.  We may instead regard it as parameterizing real, weight five, $Q$--polarized Hodge structures with Hodge numbers $(0,0,2,2,0,0)$.  We may similarly regard the other factor as parameterizing real, weight five, $Q$--polarized Hodge structures on $V_\bR \,\cap\, (V_{\s_1} \op V_{\overline\s_1})$ with Hodge numbers $(1,1,0,0,1,1)$.  Then $D$ will parameterize integral, weight five, $Q$--polarized Hodge structures on $V_\bZ$ with Hodge numbers $(1,1,2,2,1,1)$.

With the same $U_\infty$, $P_\infty$, $W$ and $N$ as in \S\ref{eg:wt1}, the cusp point now corresponds to the filtration $F_\infty \in \check D$ given by 
\[
  F_\infty^5 \ = \ V_{\s_1} \,\cap\, U_\infty  \,, \quad
  F_\infty^4 \ = \ V_{\s_1} \tand
  F_\infty^3 \ = \ V_{\s_1} \ \op \ 
  \left[(V_{\s_2} \op V_{\overline\s_2}) \,\cap\, U_\infty\right] \,.
\]
If we redefine $F \in \check D$ as 
\[
  F^5 \ = \ V_{\s_1} \,\cap\, U \,,\quad 
  F^4 \ = \ V_{\s_1} \tand
  F^3 \ = \ V_{\s_1} \ \op \  
  \left[(V_{\s_2} \op V_{\overline\s_2}) \,\cap\, U_\infty\right]\,,
\]
then the triple $(W,F,N)$ is a limiting mixed Hodge structure.

We have $\Le = \tdet(\cFe^5)^{\ot2} \ot \tdet(\cFe^4)^{\ot 2} \ot \tdet(\cFe^3)$.  The character $\chi_\infty : \monzero \to \bC^*$ is given by \eqref{E:chi-infty} as
\[
  \chi_\infty \ = \ 
  \left( \chi_\infty^5 \, \chi_\infty^4 \right)^2 \,
  \chi_\infty^3 \,,
\]
with 
\[
  \chi_\infty^p (g) \ = \ \tdet\{ g : F^p_\infty \to F^p_\infty \}
\]
for all 
\[
  g \ = \ \left[ \begin{array}{cc} a & o \\ c & \overline a{}^{-1}
  	 \end{array} \right] \ \in \ \monzero \,.
\]
We have 
\begin{eqnarray*}
  \chi_\infty^5(g) & = & \s_1(\overline a{}^{-1}) 
  \ = \ \overline\s_1(a)^{-1}\\
  \chi_\infty^4(g) & = & \overline\s_1(a)^{-1}\,\s_1(a) \\
  \chi_\infty^3(g) & = & 
  \overline\s_1(a)^{-1}\,\s_1(a)\,\s_2(a)^{-1}\,\overline\s_2(a)^{-1} \,.
\end{eqnarray*}
And \cite[\S4.4, Proposition 1]{MR265266} implies that 
\begin{eqnarray*}
  | \chi_\infty(g) | & = & 
  | \overline\s_1(a)^{-5} \,\s_1(a)^3\,\s_2(a)^{-1}\,
  	\overline\s_2(a)^{-1} |
  \ = \ |\s_1(a) \,\s_2(a) |^{-2} \ = \ 1 \,.
\end{eqnarray*}
This verifies Theorem \ref{T:normone}.

\begin{remark}[Line bundles that do not descend] \label{R:nopower}
In general, neither $\chi_\infty^5(g)$ nor $\chi_\infty^3(g)$ need have norm one.  This implies that no power of the line bundle $\tdet(\cFe^3)$ or $\tdet(\cFe^5)$ will descend to $\cP'$.  For example, suppose that matrix entry $a$ of $g$ is the unit $u \in \cO_K^*$ given by $\s_1(u) = 1 + \sqrt{2}$.  In this case we have
\begin{eqnarray}
  \nonumber
  \chi_\infty^5(g) & = & \sqrt{2}-1 \\
  \label{E:detu}
  \chi_\infty^4(g) & = & (\sqrt{2}-1)\,(\sqrt{2}+1) = 1 \\
  \nonumber
  \chi_\infty^3(g) & = & 
  (\sqrt{2}-1)\,(\sqrt{2}+1)\,(-\sqrt{2}-1)^2
  \ = \ 3+2\sqrt{2} \,.
\end{eqnarray}
So no power of $\tdet(\cFe^3)$ or $\tdet(\cFe^5)$ will descend.

Similarly, in order for some (positive) power of the line bundle $\Xi_\mathrm{e} = \tdet(\cFe^5)\ot\tdet(\cFe^4)\ot\tdet(\cFe^3)$ to descend to $\olP'$, it is necessary that the character $\chi_\infty^5\,\chi_\infty^4\,\chi_\infty^3 : \monzero \to \bC^*$ takes value in the unit circle $S^1$.  Continuing with the computations of \eqref{E:detu}, we see that $\chi_\infty^5(g)\,\chi_\infty^4(g)\,\chi_\infty^3(g) = 1-\sqrt{2}$.  So no power of $\Xi_\mathrm{e}$ will descend.
\end{remark}

\begin{remark}[The line bundle $\Le$ descends]\label{R:egdescends}
Continuing with the computation of \eqref{E:detu}, we see that 
\begin{equation}\label{E:chiu}
  \chi_\infty(g) \ = \ 
  (\sqrt{2}-1)^2 \,(1)^2\,(3+\sqrt{2}) \ = \ 1\,,
\end{equation}
when the matrix entry $a = u$.  If the matrix entry $a$ is the eighth root of unity $\zeta$ given by $\s_1(\z) = \sqrt{2}(1+\bi)/2$, then 
\begin{eqnarray}
  \nonumber
  \chi_\infty^5(g) & = & \s_1(\zeta) \\
  \label{E:detz}
  \chi_\infty^4(g) & = & \s_1(\zeta)^2 \ = \ \bi \\
  \nonumber
  \chi_\infty^3(g) & = & \bi \,;
\end{eqnarray}
and 
\begin{equation}\label{E:chiz}
  \chi_\infty(g) \ = \ 
  \s_1(\zeta)^2 \,(\bi)^2\,\bi \ = \ 1\,.
\end{equation}
One may check (for example, by a computation with Sage Math or PARI/GP) that the group of units $\cO_K^*$ in the ring of integers $\cO_K$ is generated by $u$ and $\zeta$.  Keeping Remark \ref{R:chibeta} in mind, it follows from \eqref{E:chiu} and \eqref{E:chiz} that the line bundle $\Le$ descends to $\olP'$.  Likewise, \eqref{E:detu} and \eqref{E:detz} imply that the line bundle $\tdet(\cFe^4)^{\ot4}$ descends to $\olP'$.
\end{remark}

%----------------------------------------------------------
\section{Compatibility of weight closures} \label{S:top-c}
%----------------------------------------------------------

The purpose of this section is to review compatibility properties between the weight filtrations $W^I = W(\s_I)$, and discuss some of the implications for local lifts of period maps.  These local results will have global consequences, including the following corollary of Lemma \ref{L:tedI}.

\begin{lemma}\label{L:PhiI}
The maps $\Phi_I : Z_I^* \to \Gamma_I \bs D_I$ of \eqref{E:P0I} extend to proper holomorphic maps on $Z_I \cap Z_\pi$.  These extensions are compatible with the $\Phi_J$ on $Z_J^* \subset Z_I \cap Z_\pi$ in the sense that the we have a commutative diagram
\begin{equation}\label{E:PhiIextn}
\begin{tikzcd}
  Z_J^* \arrow[r,hook] 
  \arrow[d,"\Phi_J"']
  & Z_I \cap Z_\pi
  \arrow[d,"\Phi_I"] \\
  \Gamma_J \backslash D_J \arrow[r]
  & \Gamma_I \backslash D_I \,.
\end{tikzcd}\end{equation}
\end{lemma}

%----------------------------------------------------------
\subsection{The commuting $\fsl(2)$'s} 
%----------------------------------------------------------

Our constructions are defined over the open strata $Z_I^*$, and we will need to see that these strata-wise constructions satisfy certain compatibility conditions in order to establish Lemma \ref{L:PhiI}.  The conditions are consequences of the $\tSL(2)$ orbit theorem \cite{\CKS}.  We briefly review the theorem, and then discuss consequences.

Suppose that $Z_J \subset Z_I$; equivalently, $I \subset J$.  To begin we assume that we have a local coordinate chart centered at $b \in Z_J^*$ with local monodromy cone $\s_J$ generated by $N_1,\ldots, N_k$ as in \S\ref{S:vhs}.  Given $I \subset J = \{1,\ldots,k\}$, let $\s_I$ be the face of $\s_J$ generated by the $N_i$, with $i \in I$.  Define
\[
  N_I \ = \ \sum_{i\in I} N_i 
  \tand
  N_J \ = \ \sum_{j\in J} N_j \,.
\]
Given this data, the $\tSL(2)$ orbit theorem \cite{\CKS} produces \emph{commuting $\fsl_2$--pairs}
\[
  N_I , Y_I \,;\ \hat N_J , \hat Y_J \ \in \ \fg_\bR \,.
\]
These pairs have following properties: $N_I$ and $Y_I$ commute with $\hat N_J$ and $\hat Y_J$; and there is a $(Y_I,\hat Y_J)$--eigenspace decomposition $\fg_\bC = \op\,\fg_{a,b}$, 
\[
  \fg_{a,b} \ = \ \{ \xi \in \fg_\bC \ | \ [Y_I ,\xi] = a \xi \,,\ 
  [\hat Y_J ,\xi] = b \xi \} \,,
\]
with integer eigenvalues $a,b$ that splits the weight filtrations
\begin{equation}\label{E:WIWJ}
  W_\ell (\s_I) (\fg_\bC) \ = \ \bigoplus_{a \le \ell} \fg_{a,b} 
  \tand
  W_\ell (\s_J) (\fg_\bC) \ = \ \bigoplus_{a+b\le\ell}\fg_{a,b} \,.
\end{equation}
We have
\[
  N_I \ \in \ \fg_{-2,0}
  \tand
  N_J \ \in \ \bigoplus_{a\le0} \fg_{a,-a-2} \,.
\]
If we write 
\begin{subequations}\label{SE:NJ}
\begin{equation}\label{E:NJ}
  N_J \ = \ \sum_{a\le0} N_{J,a} \,,
\end{equation}
with $N_{J,a} \in \fg_{a,-a-2}$, then 
\begin{equation}\label{E:hatNJ}
  N_{J,0} \ = \ \hat N_J \,.
\end{equation}
\end{subequations}

%----------------------------------------------------------
\subsection{Consequences for the local lifts of $\Phi_I$}
%----------------------------------------------------------

Composing the map $F_I : Z_I^* \cap \olU \to M_I$ of \eqref{E:FI} with the projection $M_I \sur D_I$ of \S\ref{S:PhiIgr} defines a local period map $\tgr(F_I) : Z_I^* \cap \olU \to D_I$.  This $\tgr(F_I)$ is the local lift of the map $\Phi_I$ defined in \S\ref{S:PhiIgr}.

Since $I \subset J$, we have $Z_J^* \subset Z_I$.  Fix a coordinate neighborhood $(t,w)\in \olU \subset \olB$ so that $Z_J^* \cap \olU =\{t=0\}$.  

\begin{lemma}\label{L:rwfp}
Suppose that $(t_m,w_m)$ and $(t_m',w_m')$ are two sequences in $Z_I^* \cap \olU$ converging to points $(0,w_\infty)$ and $(0,w'_\infty) \in Z_J^*\cap \olU$, respectively.  If $\tgr(F_I)(t_m,w_m) = \tgr(F_I)(t'_m,w_m')$ for all $m$, then $\tgr(F_J)(0,w) = \tgr(F_J)(0,w')$.
\end{lemma}

\noindent This lemma is the analog of Theorem \ref{T:top} for the local lift of $\PhiSBB$, and implies that this lift is continuous.

\begin{proof}
Given $(t,w) \in Z_I^* \cap \olU$, the map $\tgr(F_I)(t,w)$ is the composition of $F_I(t,w) = \tilde g(t,w) \cdot F$ with the projection $M_I \sur D_I$.  Moreover, $\tilde g(t,w)$ is holomorphic (and therefore continuous) on $\Delta^{k+r}$, and takes value in $C_{I,\bC}$ when restricted to $Z_I^* \cap \olU$.  So to prove the lemma, it suffices to show that 
\begin{equation}\label{E:rwfp3}
  W_\ell(\s_I)(\fc_J) \ \subset \ W_\ell(\s_J)(\fc_J) \,.
\end{equation}

The subspace $W_0(\s_J)(\fg) \subset \fg$ is the Lie algebra of the parabolic subgroup $P_{W(\s_J)} \subset G$ stabilizing the weight filtration.  So the centralizers satisfy $\fc_J \subset W_0(\s_J)(\fg_\bC)$ and $\fc_J \subset \fc_I \subset W_0(\s_I)(\fg_\bC)$.  Then \eqref{E:WIWJ} implies
\begin{equation} \label{E:rwfp1}
  \fc_J \ \subset \ \bigoplus_{\mystack{a\le0}{a+b\le0}} \fg_{a,b} \,.
\end{equation}
Note that \eqref{E:rwfp1} implies the desired \eqref{E:rwfp3} for $\ell\ge0$.  

Suppose that $X \in W_\ell(\s_I)(\fc_J)$ for some $\ell < 0$.  Then \eqref{E:WIWJ} and \eqref{E:rwfp1} imply that there exists unique $X_{a,b} \in \fg_{a,b}$ so that 
\[
  X \ = \ \sum_{\mystack{a\le\ell}{a+b\le 0}} X_{a,b} \,.
\]
In order to establish \eqref{E:rwfp3}, we need to show 
\begin{equation}\label{E:rwfp5}
   X_{a,b} \ = \ 0 \quad\hbox{for all } \ a+b > \ell \,.
\end{equation}
From $N_J(X) = 0$ and \eqref{SE:NJ} we see that $\hat N_J(X_{\ell,b}) = 0$.  Since $\{\hat N_J, \hat Y_J\}$ is an $\fsl_2$--pair, the centralizer $\fc(\hat N_J)$ of $\hat N_J$ satisfies
\begin{equation}\label{E:rwfp2}
  \fc(\hat N_J) \ \subset \ \bigoplus_{b\le0} \fg_{a,b} \,.
\end{equation}
This forces $X_{\ell,b}=0$ for all $b > 0$, and yields the desired \eqref{E:rwfp5} for $a=\ell$.

Working inductively, fix $m < \ell < 0$ and assume that \eqref{E:rwfp5} holds for all $m < a \le \ell$.   Again, $N_J(X)=0$ and \eqref{SE:NJ} implies $\hat N_J(X_{m,b})=0$ for all $m+b > \ell$.  Since $b > \ell-m > 0$, \eqref{E:rwfp2} implies $X_{m,b} = 0$ for all $m+b > \ell$.  This establishes the desired \eqref{E:rwfp5} for $a=m$ and completes the induction.
\end{proof}

%----------------------------------------------------------
\subsection{When weight filtrations coincide} \label{S:W=W}
%----------------------------------------------------------

The properties \eqref{E:WIWJ} and \eqref{E:hatNJ} yield

\begin{lemma}\label{L:W=W}
Suppose that $I \subset J$.  The following are equivalent:
\begin{i_list_emph}
\item 
The weight filtrations coincide $W(\s_I) = W(\s_J)$.
\item 
We have $\hat Y_J=0$.
\item
We have $\hat N_J=0$.
\item \label{i:c}
The cone $\s_J \subset \fc^{-1}_I$.
\end{i_list_emph}
\end{lemma}

\begin{corollary} \label{C:IW}
\begin{a_list_emph}
\item
If $I \subset I' \subset J$ and $W(\s_I) = W(\s_J)$, then $W(\s_I) = W(\s_{I'}) = W(\s_J)$.
\item
If $W(\s_{I_1}) = W(\s_{I_2})$, then $W(\s_{I_i}) = W(\s_{I_1 \cup I_2})$.
\item \label{i:ChatIc}
Given a fixed weight filtration $W$, 
\[
  I_W \ = \ \bigcup_{W(\s_I) = W} I
\]
is the unique maximal set $I_W$ such that $W = W(\s_{I_W})$.
\end{a_list_emph}
\end{corollary}

If $W(\s_I) = W(\s_J)$, then $\fg_{a,\tinyb} = \fg_{a,0}$ implies 
\begin{subequations}\label{SE:cIJ}
\begin{equation} \label{E:cIJ}
  \fc_J^{-a} \ \subset \ \fc_I^{-a} \,,
\end{equation}
and
\begin{equation}
  \frac{\fc_J^{-a}}{\fc_J^{-a-1}} \ \inj \ \frac{\fc_I^{-a}}{\fc_I^{-a-1}} \,.
\end{equation}
\end{subequations}
In the case $a=1$, the inclusion \eqref{E:cIJ} yields the striking implication (known to the experts)

\begin{lemma}\label{L:sIJ}
If $\s_J \subset \fc_I^{-1}$, then $\s_J \subset \fc_I^{-2}$.
\end{lemma}

\begin{corollary} \label{C:sIJ}
We have $\exp(\bC\s_{I_W}) \subset C^{-2}_{I,\bC}$.
\end{corollary}

%----------------------------------------------------------
\subsection{Consequences for LMHS} \label{S:EIJ}
%----------------------------------------------------------

Since $Z_J^* \subset Z_I$ if and only if $I \subset J$, we have $\Gamma_J \subset \Gamma_I$.  We will also see that $M_J \subset M_I$, cf.~\eqref{E:MJinMI}.  In particular, we have an induced $\Gamma_J\backslash M_J \to \Gamma_I \backslash M_I$.  When $W(\s_I) = W(\s_J)$, then this map descends to $\Gamma_J\backslash D_J \to \Gamma_I \backslash D_I$.  

Recall (\S\ref{S:PhiI}) that the local lift of $\Psi_I : Z_I^* \to (\Gamma_I\exp(\bC\s_I))\bs M_I$ is 
\begin{equation} \label{E:nuFI}
    \nu_I \circ F_I : Z_I^* \,\cap\, \olU \to \ 
    \exp(\bC\s_I)\backslash M_I \,.
\end{equation}
Define
\[
  Z_W \ = \ 
  \bigcup_{\mystack{Z_J \cap \olU \not=\emptyset}{W(\s_I) = W(\s_J)}}
  Z_J^* \,.
\]

\begin{lemma}\label{L:tedI}
There is a well-defined holomorphic map
\begin{equation} \label{E:tPhiI}
  \widetilde \Psi_I : Z_I \,\cap\, Z_W \,\cap\,\olU \ \to \
  \exp(\bC\s_{I_W})\backslash M_I
\end{equation}
that, when restricted to $Z_J^* \subset Z_I \cap Z_W$, coincides with the composition
$\nu_{I_W} \circ F_J$.
\end{lemma}

\begin{proof}[Proof of Lemma \ref{L:PhiI}]
Corollary \ref{C:sIJ} implies that 
\[
  (\exp(\bC\s_{I_W}) C^{-1}_{I,\bC})\backslash M_I 
  \ = \ C^{-1}_{I,\bC}\backslash M_I 
  \ = \ D_I \,.
\]
So the composition
\[\begin{tikzcd}
  Z_I \,\cap\, Z_W \,\cap\,\olU \arrow[r,"\widetilde \Psi_I"]
  & \exp(\bC\s_{I_W})\backslash M_I \arrow[r,->>]
  & (\exp(\bC\s_{I_W}) C^{-1}_{I,\bC})\backslash M_I \ = \ D_I
\end{tikzcd}\]
is the local coordinate expression for the map $\Phi_I:Z_I \cap Z_W \to \Gamma_I \backslash D_I$ of \eqref{E:PhiIextn}.  Thus Lemma \ref{L:PhiI} follows directly from Lemma \ref{L:tedI}.
\end{proof}

\begin{proof}[Proof of Lemma \ref{L:tedI}]
Suppose that $I\subset J$ and $W(\s_I) = W(\s_J)$.  Consider a local lift $\tPhi(t,w)$ over a chart $\olU$ centered at $b \in Z_J^*$ (as in \S\ref{S:vhs}).  Along 
\[
  Z_J \,\cap\, \olU \ = \ \{t_j =0 \ \forall \ j \in J\} 
  \ = \ \{0\} \times \Delta^r \ \ni \ (0,w) 
\]
we have the map $F_J : Z^*_J \cap \olU \to M_J$ of \eqref{E:FI}
\begin{subequations}\label{SE:FIJ}
\begin{equation}
  F_J(w) \ = \ \tilde g(0,w) \cdot F \,.
\end{equation}
Along $Z^*_I \cap\olU = \{ t_i =0 \ \mathrm{iff} \  i \in I \}$ we may choose a well-defined branch of $\ell(t_j)$ for all $j \in J \backslash I$.  Then the map $F_I : Z_I^* \cap \olU \to M_I$ is given by 
\begin{equation} %\label{E:FI3}
  F_I(t,w) \ = \ 
  \exp\bigg( \sum_{j \in J \backslash I} \ell(t_j) N_j\bigg) \tilde g(t,w)
  \cdot F \,.
\end{equation}
\end{subequations}
Comparing the expressions \eqref{SE:FIJ} for $F_J$ and $F_I$, and keeping $C_J \subset C_I$ and \eqref{E:cIJ} in mind, we see that 
\begin{equation} \label{E:MJinMI}
  F \ \in \ M_J \ \subset \ M_I 
\end{equation}
and $F_J$ takes value in $M_I$.  (The containment $F \in M_I$ is nontrivial, as $F$ arises from the LMHS along $Z_J^*$.)  It follows from \eqref{SE:FIJ} and \eqref{E:MJinMI} that
\[
    \nu_J \circ F_J : Z_J^* \,\cap\, \olU \ \to \ 
    \exp(\bC\s_J)\backslash M_I 
\]
\emph{also takes value in} (a quotient of) $M_I$.  The lemma now follows from \eqref{SE:FIJ}.
\end{proof}

\begin{remark}\label{R:DJinDI}
When $W(\s_I) = W(\s_J)$ and $I \subset J$, we have $C_J^{-1} = C_J \cap C_I^{-1}$.  Then \eqref{E:DI} and \eqref{E:MJinMI} imply $D_J \subset D_I$.
\end{remark}

%----------------------------------------------------------
\def\cprime{$'$} \def\Dbar{\leavevmode\lower.6ex\hbox to 0pt{\hskip-.23ex
  \accent"16\hss}D}

%----------------------------------------------------------

%----------------------------------------------------------
\end{document}

%% file: macros.tex
\def\cf{cf.~}

\newcommand{\hsp}[1]{{\hbox{\hspace{#1}}}}

\newcommand{\mystack}[2]{\ensuremath{ \substack{ \hbox{\tiny{${#1}$}} \\ \hbox{\tiny{${#2}$}} }} }
\newcounter{letcnt1} % letter section

\newcounter{letcnt2} % letter subsection
\counterwithin{letcnt2}{letcnt1}

\newcounter{talkcnt} % talk section

\def\a{\alpha}  
\def\b{\beta}  
\def\d{\delta}  
\def\e{\varepsilon}  

\def\z{\zeta}
\def\m{\mu}
\def\n{\nu}
\def\s{\sigma}

\def\x{\xi}

\def\sA{\mathscr{A}} 

 \def\sfa{\mathsf{a}}
  
\def\tAd{\mathrm{Ad}}

 \def\sB{\mathscr{B}}
 \def\sfb{\mathsf{b}}

\def\bC{\mathbb C}

\def\fc{\mathfrak{c}} \def\sfc{\mathsf{c}}

\def\fD{\mathfrak D}
 
\def\fd{\mathfrak{d}} 
\def\td{\mathrm{d}} 
 \def\tdet{\mathrm{det}}
 
 \def\tdim{\mathrm{dim}}

 \def\fe{\mathfrak{e}}

\def\cF{\mathcal F} 

\def\ff{\mathfrak{f}}

\def\tGL{\mathrm{GL}}
\def\tGr{\mathrm{Gr}}
\def\fg{{\mathfrak{g}}} 

\def\fgl{\mathfrak{gl}}
\def\tgr{\mathrm{gr}}

\def\sfH{\mathsf{H}} \def\sH{\mathscr{H}}

\def\bi{\mathbf{i}}

\def\tId{\mathrm{Id}} \def\tIm{\mathrm{Im}}
 \def\tim{\mathrm{im}}

\def\sK{\mathscr{K}}

\def\sfk{\mathsf{k}} 
 \def\tker{\mathrm{ker}}

\def\cM{\mathcal M}  
 
\def\sfm{\mathsf{m}}

 \def\cN{\mathcal N}
 
 \def\sfn{\mathsf{n}}

 \def\cO{\mathcal O}

\def\bP{\mathbb P} \def\cP{\mathcal P}
 
 \def\olP{{\overline P}}
\def\fp{\mathfrak{p}}

\def\bQ{\mathbb Q}  \def\sQ{\mathscr{Q}}

\def\bR{\mathbb R}

\def\tRe{\mathrm{Re}}

\def\sS{\mathscr{S}}

\def\bs{\mathbf{s}} 

\def\tSL{\mathrm{SL}}

\def\tStab{\mathrm{Stab}}

 \def\tspan{\mathrm{span}}
\def\fsl{\mathfrak{sl}}

  \def\tU{\mathrm{U}}
\def\sU{\mathscr{U}}

 \def\cV{\mathcal V}

\def\sfw{\mathsf{w}}
  
\def\sX{\mathscr{X}}

   \def\bZ{\mathbb Z}

\def\tand{\quad\hbox{and}\quad}
\def\bs{\backslash}

\def\tinyb{{\hbox{\tiny{$\bullet$}}}}
\def\smallb{{\hbox{\small{$\bullet$}}}}

\def\inj{\hookrightarrow}
\def\sur{\twoheadrightarrow}

\def\op{\oplus}
\def\ot{\otimes}

\def\tw{\hbox{\small $\bigwedge$}}

\newenvironment{a.list}
  {\begin{enumerate}[label=\alph*.,itemsep=3pt,leftmargin=25pt,listparindent=20pt]}
  {\end{enumerate}}
\newenvironment{a_list_emph}
  {\begin{enumerate}[label=\emph{(\alph*)},itemsep=3pt,leftmargin=25pt,listparindent=20pt]}
  {\end{enumerate}}

\newenvironment{num.list}
  {
  \begin{enumerate}[itemsep=3pt,leftmargin=25pt,listparindent=20pt,label={\arabic*.}]
  }
  {\end{enumerate}}

\newenvironment{i_list_emph}
  {\begin{enumerate}[label=\emph{(\roman*)},itemsep=3pt,leftmargin=25pt,listparindent=20pt]}
  {\end{enumerate}}

%% file: thms.tex
\newtheorem{corollary}[equation]{Corollary}
\newtheorem{lemma}[equation]{Lemma}
\newtheorem{proposition}[equation]{Proposition}
\newtheorem{theorem}[equation]{Theorem}
\newtheorem{conjecture}[equation]{Conjecture}
\newtheorem*{mainthm*}{Main Theorem}
\newtheorem*{theorem*}{Theorem}
\newtheorem*{corollary*}{Corollary}
\newtheorem*{lemma*}{Lemma}

\theoremstyle{definition}

\newtheorem*{boldQ*}{Question}
\newtheorem*{boldP*}{Problem}

\theoremstyle{definition}

\theoremstyle{remark}
\newtheorem*{assume*}{Assume}
\newtheorem*{answer*}{Answer}

\newtheorem*{claim*}{Claim}

\newtheorem*{definition*}{Definition}

\newtheorem*{example*}{Example}

\newtheorem*{hint*}{Hint}
\newtheorem*{notation*}{Notation}
\newtheorem{remark}[equation]{Remark}
\newtheorem*{remark*}{Remark}
\newtheorem*{remarks*}{Remarks}
\newtheorem*{fact*}{Fact}

\newtheorem*{emphQ*}{Question}
\newtheorem*{emphA*}{Answer}